\documentclass[reqno,a4paper,11pt]{article}
\usepackage{amsmath,amsthm,dsfont,amsfonts,amssymb,fancyhdr}
\usepackage{enumerate}
\usepackage[usenames,dvipsnames]{color}
\usepackage{bbm,soul,supertabular,longtable,verbatim,extarrows}
\usepackage{titlesec}
\usepackage{graphicx}
\usepackage[titletoc,toc,title]{appendix}

\usepackage{tikz}
\usetikzlibrary{arrows,decorations.pathmorphing,backgrounds,positioning,fit,petri}

\usepackage{caption}

\usepackage{float}
\usepackage{BOONDOX-cal}
\usepackage{tikz}
\usepackage{booktabs,multirow,makecell}
\usepackage{authblk}
\usepackage[titletoc]{appendix}
\usetikzlibrary[trees]

\usepackage{lipsum}
\usepackage{mathrsfs}
\usepackage{indentfirst}
\usepackage[colorlinks=true, allcolors=blue]{hyperref}
\usepackage[top=2.4cm,bottom=2.2cm,left=2.6cm,right=2cm]{geometry}

\newtheorem{thm}{Theorem}[section]
\newtheorem{lem}[thm]{Lemma}

\newtheorem{de}[thm]{Definition}
\newtheorem{re}[thm]{Remark}

\newtheorem{cla}[thm]{Claim}

\soulregister\cite7
\soulregister\citep7
\soulregister\citet7
\soulregister\ref7
\soulregister\pageref7
\sethlcolor{yellow}
\setstcolor{red}
\soulregister{\em}{0}

\begin{document}
	\title{Sesqui-regular graphs with smallest eigenvalue at least $-3$}
	\author[a]{Qianqian Yang}
	\author[a]{Brhane Gebremichel}
	\author[b]{Masood Ur Rehman}
	\author[c]{Jae Young Yang}
    \author[a,d]{Jack H. Koolen\thanks{Corresponding author}}
	\affil[a]{\footnotesize{School of Mathematical Sciences, University of Science and Technology of China, 96 Jinzhai Road, Hefei, 230026, Anhui, PR China}}
	\affil[b]{\footnotesize{Department of Basic Sciences, Balochistan University of Engineering and Technology Khuzdar, Khuzdar 89100, Pakistan}}
	\affil[c]{\footnotesize{Samsung SDS, Olympic-ro 35-gil 125, Songpa-gu, Seoul, 05510, Republic of Korea}}
	\affil[d]{\footnotesize{CAS Wu Wen-Tsun Key Laboratory of Mathematics, University of Science and Technology of China, 96 Jinzhai Road, Hefei, Anhui, 230026, PR China}}
	
	\maketitle
	\pagestyle{plain}
	
	\newcommand\blfootnote[1]{%
		\begingroup
		\renewcommand\thefootnote{}\footnote{#1}%
		\addtocounter{footnote}{-1}%
		\endgroup}
	\blfootnote{2010 Mathematics Subject Classification. 05C50, 05C62, 05C75, 11H99} 
	\blfootnote{E-mail addresses:  {\tt qqyang91@ustc.edu.cn} (Q. Yang),  {\tt brhaneg220@mail.ustc.edu.cn} (B. Gebremichel), {\tt masoodqau27@gmail.com} (M.U. Rehman), {\tt piez@naver.com} (J.Y. Yang), {\tt koolen@ustc.edu.cn} (J.H. Koolen).}

\begin{abstract}
	Koolen et al. showed that if a graph with smallest eigenvalue at least $-3$ has large minimal valency, then it is $2$-integrable. In this paper, we will focus on the sesqui-regular graphs with smallest eigenvalue at least $-3$ and study their integrability. 
\end{abstract}

\section{Introduction}
In this paper, all graphs considered are finite, undirected and simple, and the eigenvalues of graphs we mean are the eigenvalues of their adjacency matrix. 

Let $s$ be a positive integer and $\Gamma$ a graph with smallest eigenvalue $\lambda_{\min}(\Gamma)$. We say that the graph $\Gamma$ is \emph{$s$-integrable}, if for each vertex $x$ of $\Gamma$, there exists an integral vertex $\mathbf{x}$ such that the following equations hold:
\begin{equation}\label{def: s-integrable}
(\mathbf{x},\mathbf{y})=\left\{
\begin{array}{ll}
	s\lceil-\lambda_{\min}(\Gamma)\rceil & \text{ if }x=y, \\
	s & \text{ if $x$ is adjacent to $y$}, \\
	0 & \text{ otherwise},
\end{array}
\right.	
\end{equation}
where $(,)$ denotes the standard inner product. Note that if a graph is $s$- and $t$-integrable, then it is also $(s+t)$-integrable.

In 1976,  Cameron et al. \cite{cameron} showed that: 
\begin{thm}\label{thm:cameron}
Any connected graph $\Gamma$ with smallest eigenvalue $\lambda_{\min}(\Gamma)\geq-2$ is $s$-integrable for any integer $s\geq2$. Moreover, if $\Gamma$ has more than $36$ vertices, then $\Gamma$ is $1$-integrable.  
\end{thm}

Later in 2018, Koolen et al. \cite{kyy3} studied the graphs with smallest eigenvalue at least $-3$ and proved that:
\begin{thm}
	There exists a positive constant $\kappa$ such that any graph $\Gamma$ with smallest eigenvalue $\lambda_{\min}(\Gamma)\geq-3$ and minimal valency $k_{\min}(\Gamma)\geq\kappa$ is $s$-integrable, for any integer $s\geq2$. 
\end{thm}

In this paper, we will continue the work in \cite{kyy3} and study the integrability of a class of special graphs with smallest eigenvalue at least $-3$. 

Suppose that $\Gamma$ is a graph with vertex set $V(\Gamma)$. For any two vertices $x$ and $y$, if $x$ and $y$ are adjacent, we write
$x\sim y$, and $x\not\sim y$ otherwise. Let $x$ be a vertex of $\Gamma$. We denote by $\Gamma_i(x)$ the set of vertices at distance $i$ from $x$. If a vertex $y$ is in $\Gamma_i(x)$, then we denote the distance between $x$ and $y$ by $d(x,y)$ and write $d(x,y)=i$.
 If $i=1$, we say the value $|\Gamma_1(x)|$ the \emph{valency} of $x$. Whenever the valency of every vertex of $\Gamma$ is the same, $k$ say, $\Gamma$ is also said to be \emph{regular} of valency $k$ or \emph{$k$-regular}.  

A regular graph $\Gamma$ with $n$ vertices and valency $k$ is \emph{sesqui-regular} with parameters $(n,k,c)$, if any two vertices at distance $2$ have exactly $c$ common neighbors.

A regular graph $\Gamma$ with $n$ vertices and valency $k$ is \emph{strongly regular} with parameters $(n,k,a,c)$, if any two adjacent vertices have exactly $a$ common neighbors and any two distinct non-adjacent vertices have exactly $c$ common neighbors. Here is a family of important strongly regular graphs constructed from Steiner triple systems.

A \emph{Steiner triple system} $STS(v)$ is a pair $(\mathcal{P},\mathcal{B})$, where $\mathcal{P}$ is a
$v$-element set called a \emph{point set} and $\mathcal{B}$ is a family of $3$-element subsets of $\mathcal{P}$ called a \emph{block set} such that each $2$-element subset of $\mathcal{P}$ is contained in exactly one block. Take a Steiner triple system $STS(v)$. (A necessary and sufficient condition for the existence of an $STS(v)$ is that $v\equiv 1,3\pmod 6$ \cite{Kirkman}.) Construct a graph as follows: its vertex set is the block set of this Steiner triple system, and two blocks are adjacent whenever they intersect in one point. This gives a strongly regular graph with parameters $(\frac{v(v-1)}{6},\frac{3(v-3)}{2},\frac{v+3}{2}, 9)$ (see \cite[p.~5]{Cioaba.2014}).
We call this graph the \emph{block graph} of the Steiner triple system. It is not hard to see that this graph is $1$-integrable with smallest eigenvalue $-3$.

Recently, Koolen et al. showed the following result for sesqui-regular graphs with large minimal valency and fixed smallest eigenvalue.

\begin{thm}[{cf.~\cite[Theorem 1.5]{Koolen.2019}}]\label{kgyy}
Let $\lambda \geq 2$ be an integer. There exists a real number $K_1(\lambda)$ such that, for any connected sesqui-regular graph $\Gamma$ with parameters $(v, k, c)$ and smallest eigenvalue $\lambda_{\min}(\Gamma)\geq-\lambda$, if $k \geq K_1(\lambda)$, then either $c\leq \lambda^2(\lambda-1)$ or $v-k-1\leq \frac{(\lambda -1)^2}{4}+1$ holds.
\end{thm}

In the present paper, we will improve Theorem \ref{kgyy} as follows for sesqui-regular graphs with smallest eigenvalue at least $-3$.

\begin{thm}\label{new}
	There exists a positive constant $K_2~(\geq11)$ such that, for any connected sesqui-regular graph $\Gamma$ with parameters $(n,k,c)$, where $c\geq9$, and smallest eigenvalue $\lambda_{\min}(\Gamma)\in[-3,-2)$, if $k\geq K_2$, then one of the following holds:
	\begin{enumerate}
		\item $c=9$ and $\Gamma$ is the block graph of a Steiner triple system;
		\item $c=k-1$ and $\Gamma$ is the complement of the disjoint union of cycles with length at least $4$;
		\item $c=k$ and $\Gamma$ is the complete multipartite graph $K_{\frac{n}{3}\times 3}$, that is, the complement of the disjoint union of cycles with length $3$.
	\end{enumerate}
In particular, in all these three cases, $\Gamma$ is $1$-integrable.
\end{thm}

\begin{re}
Let $\lambda\geq2$ be an integer. Let $\Gamma$ be a strongly regular graph with  parameters $(n,k,a,c)$ and smallest eigenvalue $-\lambda$. Neumaier \cite[Theorem 5.1]{Neumaier.1979} showed that, if $k$ is sufficiently large, then $\Gamma$ is a Steiner graph, a Latin square graph, or a complete multipartite graph. This implies that for any strongly reguar graph with parameters $(n,k,a,c)$ and fixed smallest eigenvalue $-\lambda$, if its valency is very large, then either $c=k$ or $c\leq \lambda^2$ holds. Therefore, Theorem \ref{new} can be regarded as a generalization of this result of Neumaier in the case where $\lambda=3$. Note that in our proof, we do not use association schemes as our tool.
\end{re}

In order to show Theorem \ref{new}, it is sufficient to show the following results.

\begin{thm}\label{c equals k or k-1}
Let $\Gamma$ be a connected sesqui-regular graph with parameters $(n,k,c)$ and smallest eigenvalue $\lambda_{\min}(\Gamma)\in[-3,-2)$. 
\begin{enumerate}
	\item If $c=k$, then $n\geq 6$ and $\Gamma$ is the complete multipartite graph $K_{\frac{n}{3}\times 3}$, that is, the complement of the disjoint union of cycles with length $3$.
	\item If $c=k-1$, then $n\geq 7$ and $\Gamma$ is the complement of the disjoint union of cycles with length at least $4$, or is the $3$-cube.
	\end{enumerate}
In particular, $\Gamma$ is $1$-integrable when $c=k$ or $k-1$.
\end{thm}

\begin{thm}\label{1 integrability}
There exists a positive constant $K_3$ such that for any connected sesqui-regular graph $\Gamma$ with parameters $(n,k,c)$, where $c\geq9$, and smallest eigenvalue $\lambda_{\min}(\Gamma)\in[-3,-2)$, if $k\geq K_3$, then $\Gamma$ is $1$-integrable.
\end{thm}

\begin{thm}\label{c equals 9}
Let $\Gamma$ be a connected sesqui-regular graph with parameters $(n,k,c)$, where $k-2\geq c\geq9$, and smallest eigenvalue $\lambda_{\min}(\Gamma)\in[-3,-2)$. If $\Gamma$ is $1$-integrable, then $c=9$ and $\Gamma$ is the block graph of a Steiner triple system.	
\end{thm}

%
%

We will proceed as follows: In Section \ref{proof of the case where c equals k or k-1}, we give the proof of Theorem \ref{c equals k or k-1}. As we need Hoffman graphs to show Theorem \ref{1 integrability}, we first prove Theorem \ref{c equals 9} in Section \ref{proof of the case where c is at least k-2}. In the last section, we will introduce the terminology related to Hoffman graphs and show Theorem \ref{1 integrability}.

\section{Proof of Theorem \ref{c equals k or k-1}}\label{proof of the case where c equals k or k-1}

Let $\Gamma$ be a connected sesqui-regular graph with parameters $(n,k,c)$ and smallest eigenvalue $\lambda_{\min}(\Gamma)$. Assume
\[-3\leq\lambda_{\min}(\Gamma)<-2.\]
We immediately find that $k\geq3$ and $c\geq1$. In this section, we deal with the special case where $k-1\leq c\leq k$ and prove Theorem \ref{c equals k or k-1}. First, let us look at the case where $c=k$.

\begin{lem}\label{c is k}
	Suppose that $\Gamma$ is a connected sesqui-regular graph with parameters $(n,k,k)$ and smallest eigenvalue $\lambda_{\min}(\Gamma)\in [-3,2)$. Then $n\geq6$ and $\Gamma$ is the complete multipartite graph $K_{\frac{n}{3}\times 3}$.
\end{lem}
\begin{proof}
	As the parameters of $\Gamma$ is $(n,k,k)$, any two vertices at distance $2$ in $\Gamma$ have the same neighbors. This implies that the diameter of $\Gamma$ is $2$, and for any two vertices $x$ and $y$ of $\Gamma$, $\Gamma_1(x)=\Gamma_1(y)$ holds if $x\not\sim y$. Now we define a relation $\mathcal{R}$ on the vertex set of $\Gamma$ such that for any two vertices $x$ and $y$, $x\mathcal{R}y$ if and only if $\Gamma_1(x)=\Gamma_1(y)$. It is not hard to find that $\mathcal{R}$ is an equivalence relation, and for any two distinct vertices $x$ and $y$, $x\mathcal{R}y$ if and only if $x\not\sim y$. This means that the vertices in the same equivalence class are pairwise non-adjacent, and the vertices in different equivalence classes are always adjacent. Considering that $\Gamma$ is regular, we conclude that all of these equivalence classes have the same cardinality, and hence $\Gamma$ is the complete multipartite graph $K_{\frac{n}{p}\times p}$ for some integer $p$. We have $\frac{n}{p}\geq2$ immediately, as $\Gamma$ is connected. Note that the smallest eigenvalue of the graph $K_{\frac{n}{p}\times p}$ equals $-p$. We have $p=3$, as $\lambda_{\min}(\Gamma)\in [-3,2)$. Therefore, the lemma holds.
\end{proof}

Before discussing the case where $c=k-1$, the concept of eigenvalue interlacing as follows is needed.

\begin{lem}\label{lem: interlacing}
	\begin{enumerate}
		\item {\rm (cf.~\cite[Theorem 9.1.1]{godsil})} Suppose that $\Gamma$ is a graph and $\Gamma'$ is an induced subgraph of $\Gamma$, then the smallest eigenvalue of $\Gamma'$ is at least the smallest eigenvalue of $\Gamma$.
		\item {\rm (cf.~\cite[Lemma  9.6.1]{godsil})} Suppose that $\Gamma$ is a graph and $\pi$ is a partition of its vertex set $V(\Gamma)$ with cells $V_1,\ldots,V_r$. Define the \emph{quotient matrix} of $\Gamma$ relative to $\pi$ to be the $r\times r$ matrix $B$ such that the $ij$-entry of $B$ is the average number of neighbors in $V_j$ of vertices in $V_i$. Then the smallest eigenvalue of $B$ is at least the eigenvalue of $\Gamma$.
	\end{enumerate}
\end{lem}

Now we are going to prove the following lemma.

\begin{lem}\label{c is k-1}
	Suppose that $\Gamma$ is a connected sesqui-regular graph with parameters $(n,k,k-1)$, smallest eigenvalue $\lambda_{\min}(\Gamma)\in [-3,-2)$ and diameter $D(\Gamma)$.
	\begin{enumerate}
		\item If $D(\Gamma)=2$, then $n\geq 7$ and $\Gamma$ is the complement of the disjoint union of cycles with length at least $4$.
		\item If $D(\Gamma)\geq3$, then $D(\Gamma)=3$ and $\Gamma$ is the $3$-cube.
		\end{enumerate}
\end{lem}
\begin{proof}
	{\rm (i)} Assume $D(\Gamma)=2$. Since $\Gamma$ is sesqui-regular with parameters $(n,k,k-1)$, we have, for any two distinct vertices $w_1$ and $w_2$,
	\begin{equation}\label{eq: c=k-1 d=2 0}
		|\Gamma_1(w_1)\cap\Gamma_1(w_2)|=k-1 \text{ if }w_1\not\sim w_2.
	\end{equation}
	\noindent Pick a vertex $x$ of $\Gamma$. It is not hard to see $\Gamma_2(x)\neq\emptyset$, as otherwise $k=n-1$ and $G$ is the complete graph $K_n$ which has smallest eigenvalue $-1$. This leads to a contradiction, as $\lambda_{\min}(\Gamma)\in[-3,-2)$. By \eqref{eq: c=k-1 d=2 0}, we have, for any $y\in\Gamma_2(x)$,
	\begin{equation}\label{eq: c=k-1 d=2}
		|\Gamma_1(y)\cap\Gamma_1(x)|=k-1 \text{ and }
		|\Gamma_1(y)\cap\Gamma_2(x)|=1.
	\end{equation}
	This implies that the induced subgraph of $\Gamma$ on $\Gamma_2(x)$ is the disjoint union of $K_2$'s. We claim $|\Gamma_2(x)|=2$. Otherwise, we can find $\{y_1,y_1',y_2,y_2'\}\subseteq\Gamma_2(x)$ with $y_1\sim y_1'$ and $y_2\sim y_2'$. Let $\Gamma'$ be the induced subgraph of $\Gamma$ on the set $\{x\}\cup \Gamma_1(x)\cup\{y_1,y_1',y_2,y_2'\}$. Then, it is straightforward to check that the quotient matrix of $\Gamma'$ relative to the partition $\{\{x\},\Gamma_1(x),\{y_1,y_1',y_2,y_2'\}\}$, by \eqref{eq: c=k-1 d=2}, is as follows
	\[
	\left(
	\begin{array}{ccc}
		0 & k & 0 \\
		1 & k-1-\frac{4(k-1)}{k} & \frac{4(k-1)}{k} \\
		0 & k-1 & 1 \\
	\end{array}
	\right),
	\]
	
	\noindent which has smallest eigenvalue $-\frac{2(k-1)+\sqrt{5k^2-8k+4}}{k}$. By Lemma \ref{lem: interlacing},  $-\frac{2(k-1)+\sqrt{5k^2-8k+4}}{k}\geq\lambda_{\min}(\Gamma')\geq\lambda_{\min}(\Gamma)\geq-3$ holds. This shows $k=3$. But as $y_1\not\sim y_2$, we have $|\Gamma_1(y_1)\cap\Gamma_1(y_2)|=k-1=2$ by \eqref{eq: c=k-1 d=2 0}. Note that $\Gamma_1(y_1)=\{y_1'\}\cup(\Gamma_1(y_1)\cap\Gamma_1(x))$ and $|\Gamma_1(y_1)\cap\Gamma_1(x)|=2$ by \eqref{eq: c=k-1 d=2}. It follows $\Gamma_1(y_1)\cap\Gamma_1(y_2)= \Gamma_1(y_1)\cap\Gamma_1(x)$, as $y_1'\not\sim y_2$. Similarly, we have $\Gamma_1(y_1)\cap\Gamma_1(y_2')= \Gamma_1(y_1)\cap\Gamma_1(x)$. Let $z\in\Gamma_1(y_1)\cap\Gamma_1(x)$. Then $\{x,y_1,y_2,y_2'\}\subseteq\Gamma_1(z)$, which contradicts $|\Gamma_1(z)|=k=3$. Hence $|\Gamma_2(x)|=2$ and the induced subgraph of $\Gamma$ on $\Gamma_2(x)$ is a $K_2$. This reveals that the complement $\bar{\Gamma}$ of $\Gamma$ is $2$-regular and the two neighbors of each vertex in $\bar{\Gamma}$ are non-adjacent. In other words, $\bar{\Gamma}$ is a disjoint union of cycles with length at least $4$. Note that $\lambda_{\min}(\Gamma)\in[-3,-2)$ implies that the second largest eigenvalue of $\bar{\Gamma}$ is greater than $1$. Hence, either $\bar{\Gamma}$ is the complement of a cycle with length at least $7$ or the complement of the disjoint unision of at least $2$ cycles with length at least $4$. This shows that $n\geq7$. 
	
	{\rm (ii)}	Assume $D(\Gamma)\geq3$. Let $x_0\sim x_1\sim x_2\sim x_{3}$ be a path where $x_i\in\Gamma_i(x_0)$ for $i=1,2,3$ and let $\Gamma''$ be the subgraph of $\Gamma$ induced on the set $\{x_0\}\cup\Gamma_1(x_0)\cup(\Gamma_2(x_0)\cap\Gamma_1(x_3))\cup\{x_3\}$. It follows immediately that
	\begin{itemize}
		\item[(a)] for any vertex $w_1\in \Gamma_2(x_0)\cap\Gamma_1(x_3)$, it has exactly $k-1$ neighbors in the set $\Gamma_1(x_0)$, as $d(x_0,w_1)=2$.
	\end{itemize}
We claim that $d(w_2,x_3)=2$ for any vertex $w_2\in \Gamma_1(x_0)$, which will be shown at the end of the proof. Thus, 
\begin{itemize}
	\item[(b)] for any vertex $w_2\in \Gamma_1(x_0)$, it also has exactly $k-1$ neighbors in the set $\Gamma_2(x_0)\cap\Gamma_1(x_3)$ as $d(w_2,x_3)=2$.
\end{itemize}
By double-counting the edges between the two sets $\Gamma_1(x_0)$ and $\Gamma_2(x_0)\cap\Gamma_1(x_3)$, we have $|\Gamma_2(x_0)\cap\Gamma_1(x_3)|=|\Gamma_1(x_0)|=k$ and thus $\Gamma_2(x_0)\cap\Gamma_1(x_3)=\Gamma_1(x_3)$. Until now, it is not hard to see that every vertex in $\Gamma''$ has valency $k$ and we conclude $\Gamma''=\Gamma$, as $\Gamma$ is connected. From (a) and (b), we find that the quotient matrix of $\Gamma$ relative to the partition $\{\{x_0\},\Gamma_1(x_0),\Gamma_2(x_0)\cap\Gamma_1(x_3),\{x_3\}\}$ is the $4\times4$ matrix
	\begin{equation*}\label{eq: c=k-1 d=2 quotient}
		\left(
		\begin{array}{cccc}
			0 & k & 0 & 0 \\
			1 & 0 & k-1 & 0 \\
			0 & k-1 & 0 & 1 \\
			0 & 0 & k & 0 \\
		\end{array}
		\right)
	\end{equation*}
	with smallest eigenvalue $-k$. By Lemma \ref{lem: interlacing} {\rm (ii)}, we have $-k\geq\lambda_{\min}(\Gamma)\geq-3$ and hence $k=3$. Using (a) and (b) again, we conclude that $\Gamma$ is the $3$-cube.
	
Now we prove our claim.	As $d(x_1,x_3)=2$, we have $2\leq k-1=|\Gamma_1(x_1)\cap \Gamma_1(x_3)|\leq |\Gamma_2(x_0)\cap\Gamma_1(x_3)|$. So we can find two distinct vertices $w_3$ and $w_4$ in  $\Gamma_2(x_0)\cap\Gamma_1(x_3)$. From (a) we obtain, for $i=3,4$, that
\begin{equation}\label{eq:distance 2 1}
	|\Gamma_1(w_i)\cap\Gamma_1(x_0)|=k-1,
	\end{equation}
and futher 
\begin{equation}\label{eq:distance 2 2}
	\Gamma_1(w_i)=(\Gamma_1(w_i)\cap\Gamma_1(x_0))\cup\{x_3\}.
\end{equation}
This tells us $w_3\not\sim w_4$. Note that $d(w_3,w_4)=2$. Thus, $k-1=|\Gamma_1(w_3)\cap\Gamma_1(w_4)|=|(\Gamma_1(w_3)\cap\Gamma_1(w_4)\cap\Gamma_1(x_0))\cup\{x_3\}|$, that is, 
\begin{equation}\label{eq:distance 2 3}
	|\Gamma_1(w_3)\cap\Gamma_1(w_4)\cap\Gamma_1(x_0)|=k-2.  
\end{equation}
By the inclusion-exclusion principle and \eqref{eq:distance 2 1} and \eqref{eq:distance 2 3}, we have 
\[\begin{split}
	|(\Gamma_1(w_3)\cup\Gamma_1(w_4))\cap\Gamma_1(x_0)|&=	|(\Gamma_1(w_3)\cap\Gamma_1(x_0))\cup(\Gamma_1(w_4)\cap\Gamma_1(x_0))|\\
	&=|\Gamma_1(w_3)\cap\Gamma_1(x_0)|+|\Gamma_1(w_4)\cap\Gamma_1(x_0)|-|\Gamma_1(w_3)\cap\Gamma_1(w_4)\cap\Gamma_1(x_0)|\\
	&=2(k-1)-(k-2)\\
	&=k=|\Gamma_1(x_0)|.
	\end{split}
\]
Hence, $\Gamma_1(x_0)\subset\Gamma_1(w_3)\cup\Gamma_1(w_4)$ and the claim holds. The proof is completed.
\end{proof}

\begin{proof}[\rm\textbf{Proof of Theorem \ref{c equals k or k-1}}]To prove this theorem, it is sufficient to show that the complement of the disjoint union of cycles and the $3$-cube are $1$-integrable by Lemma \ref{c is k} and Lemma \ref{c is k-1}.
	
Assume that $\Gamma$ is the complement of disjoint union of $m$ cycles $C_1, C_2,\ldots, C_m$, where the length of $C_i$ is $\ell_i$ and $V(C_i)=\{x_{i,1},\ldots,x_{i,\ell_i}\}$,  and that two vertices $x_{i,p}$ and $x_{i,q}$ are adjacent in the cycle $C_i$ if and only if $p-q\equiv \pm1\pmod {\ell_i}$. For each vertex $x_{i,p}$ of $\Gamma$, define the integral vector $\mathbf{x}_{i,p}$ as follows:
\[
\mathbf{x}_{i,p}=\left\{
\begin{array}{ll}
 \mathbf{e}+\mathbf{e}_{i,p}-\mathbf{e}_{i,p+1} &\text{ if }p<\ell_i,\\
 \mathbf{e}+\mathbf{e}_{i,\ell_i}-\mathbf{e}_{i,1}& \text{ if }p=\ell_i,
\end{array}
\right.
\]
where $\{\mathbf{e},\mathbf{e}_{i,p_i}\mid i=1,\ldots,m,p_i=1,\ldots,\ell_i\}$ is a set of orthonormal integral vectors. This shows that $\Gamma$ is $1$-integrable.

As for the case where $\Gamma$ is the $3$-cube, it is straightforward to check that $\Gamma$ is $1$-integrable. This completes the proof. 
\end{proof}

\section{Proof of Theorem \ref{c equals 9}} \label{proof of the case where c is at least k-2}
In this section, we will mainly prove Theorem \ref{c equals 9}. So we will always assume, in this section, that $\Gamma$ is a connected $1$-integrable sesqui-regular graph with parameters $(n,k,c)$, where $c\leq k-2$, and smallest eigenvalue $\lambda_{\min}(\Gamma)\in[-3,-2)$, and that $\{\mathbf{x}\mid x\in V(\Gamma)\}$ is a family of integral vectors which satisfy \eqref{def: s-integrable} with $s=1$. 

Assume $\{\mathbf{x}\mid x\in V(\Gamma)\}\subset \mathbb{Z}^m$ for some integer $m$. For $i=1,\ldots,m$, let $\mathbf{e}_{i}$ be the vector whose $i$-th entry is $1$ and the others are zero. Since the norm of the integral vector $\mathbf{x}$ is $3$, $\mathbf{x}$ is a $(0,\pm1)$-vector, that is, there exist $i_1,i_2,i_3\in\{1,2,\ldots,m\}$ and $a_{i_1},a_{i_2},a_{i_3}\in\{1,-1\}$ such that $\mathbf{x}$ can be uniquely expressed as \[\mathbf{x}=a_{i_1}\mathbf{e}_{i_1}+a_{i_2}\mathbf{e}_{i_2}+a_{i_3}\mathbf{e}_{i_3}.\] We call the set $\{i_1,i_2,i_3\}$ the \emph{support} of $\mathbf{x}$, and denote it by $\operatorname{supp}(\mathbf{x})$. Also, we define a mapping $\sigma_x:\{1,2,\ldots,m\}\to\{0,\pm1\}$ such that 
	\[
\sigma_x(i)=\left\{
\begin{array}{ll}
	a_i & \text{if } i\in \operatorname{supp}(\mathbf{x}), \\
	0 & \text{otherwise}.
\end{array}
\right.
\]
Let $x$ and $y$ be two distinct vertices of $\Gamma$. It follows that
\begin{equation}\label{eq inner product}
	(\mathbf{x},\mathbf{y})=\sum\nolimits_{\ell=1}^m\sigma_x(\ell)\sigma_y(\ell)=\sum\nolimits_{\ell\in \operatorname{supp}(\mathbf{x})\cap\operatorname{supp}(\mathbf{y})}\sigma_x(\ell)\sigma_y(\ell)=\left\{
	\begin{array}{ll}
		1 & \text{if } x\sim y, \\
		0 & \text{if } x\not\sim y.
	\end{array}
	\right.
	\end{equation}
Remark that if $\operatorname{supp}(\mathbf{x})\cap\operatorname{supp}(\mathbf{y})=\{i\}$, we will write $\mathbf{e}_{\operatorname{supp}(\mathbf{x})\cap\operatorname{supp}(\mathbf{y})}$ instead of $\mathbf{e}_{i}$  for convenience.

There is an easy observation as follows:
\begin{lem}\label{lem: general common support}
For any two distinct vertices $x$ and $y$ of $\Gamma$, $|\operatorname{supp}(\mathbf{x})\cap\operatorname{supp}(\mathbf{y})|=1$ or $3$ holds if $x\sim y$, and 
$|\operatorname{supp}(\mathbf{x})\cap\operatorname{supp}(\mathbf{y})|=0$ or $2$ holds if $x\not\sim y$.
\end{lem}
\begin{proof}
It follows straightforward from \eqref{eq inner product} .	
\end{proof} 
Let 
\[
S=\{x\in V(\Gamma)\mid \text{ there exists }x'\neq x\text{ with }\operatorname{supp}(\mathbf{x'})=\operatorname{supp}(\mathbf{x})\}.
\]

\begin{lem}\label{lem: 3 common neighbor}
If $S\neq\emptyset$, then for a given vertex $x\in S$, there exists a unique vertex $x'\neq x$ such that $\operatorname{supp}(\mathbf{x'})=\operatorname{supp}(\mathbf{x})$ holds. In particular, $x'$ is adjacent to $x$.
\end{lem}
\begin{proof}
	Assume  $\mathbf{x}=a_{i_1}\mathbf{e}_{i_1}+a_{i_2}\mathbf{e}_{i_2}+a_{i_3}\mathbf{e}_{i_3}$, where $a_{i_1},a_{i_2},a_{i_3}\in\{1,-1\}$. If a vertex $y$ distinct from $x$ satisfies $\operatorname{supp}(\mathbf{y})=\operatorname{supp}(\mathbf{x})=\{i_1,i_2,i_3\}$, then $y\sim x$ by Lemma \ref{lem: general common support} and hence $(\mathbf{y},\mathbf{x})=1$. This implies that $\mathbf{y}$ equals one of the following vectors:
	\begin{equation}\label{eq: 3 common neighbor}
-a_{i_1}\mathbf{e}_{i_1}+a_{i_2}\mathbf{e}_{i_2}+a_{i_3}\mathbf{e}_{i_3},~a_{i_1}\mathbf{e}_{i_1}-a_{i_2}\mathbf{e}_{i_2}+a_{i_3}\mathbf{e}_{i_3},~a_{i_1}\mathbf{e}_{i_1}+a_{i_2}\mathbf{e}_{i_2}-a_{i_3}\mathbf{e}_{i_3}.
	\end{equation}
	Since any two distinct vectors in \eqref{eq: 3 common neighbor} have inner product $-1$, we obtain the uniqueness.  
\end{proof}

For each vertex $x\in S$, we call the unique vertex $x'$ with $\operatorname{supp}(\mathbf{x'})=\operatorname{supp}(\mathbf{x})$ the \emph{mate} of $x$.

\begin{lem}\label{lem: neighbor 1 common}
	Given a vertex $x$, 
	\begin{enumerate}
		\item if a vertex $y$ is adjacent to $x$ but not the mate of $x$, then $|\operatorname{supp}(\mathbf{x})\cap\operatorname{supp}(\mathbf{y})|=1$ and 
		$(\mathbf{x},\mathbf{e}_{\operatorname{supp}(\mathbf{x})\cap\operatorname{supp}(\mathbf{y})})=(\mathbf{y},\mathbf{e}_{\operatorname{supp}(\mathbf{x})\cap\operatorname{supp}(\mathbf{y})})$.
		\item if $\{i_1,i_2\}\subset \operatorname{supp}(\mathbf{x})$, then $|\{y\in V(\Gamma)\mid (\mathbf{y},\sigma_x(i_1)\mathbf{e}_{i_1}+\sigma_x(i_2)\mathbf{e}_{i_2})=2\}|\leq 2$, with the equality holds if and only if $x\in S$.
	\end{enumerate}
\end{lem}

\begin{proof}
{\rm (i)} This is straightforward by Lemma \ref{lem: general common support}, since $x$ and $y$ are not mates.

{\rm (ii)} Assume $\operatorname{supp}(\mathbf{x})=\{i_1,i_2,i_3\}$ and $\mathbf{x}=a_{i_1}\mathbf{e}_{i_1}+a_{i_2}\mathbf{e}_{i_2}+a_{i_3}\mathbf{e}_{i_3}$.
If $y$ is a vertex distinct from $x$ with  $(\mathbf{y},a_{i_1}\mathbf{e}_{i_1}+a_{i_2}\mathbf{e}_{i_2})=2$, then  
$\mathbf{y}=a_{i_1}\mathbf{e}_{i_1}+a_{i_2}\mathbf{e}_{i_2}-a_{i_3}\mathbf{e}_{i_3}$ as $(\mathbf{x},\mathbf{y})\in\{0,1\}$. In other words, $y$ is the mate of $x$ and $x\in S$. This shows {\rm (ii)}.
\end{proof}

For the set $S$, we claim the following.
\begin{cla}\label{cla: no same support}
$S=\emptyset$ holds if $c\geq 9$.
\end{cla}
In order to make the proof of Theorem \ref{1 integrability} more clear and logical, we will assume $S=\emptyset$ until Theorem \ref{1 integrability} is proved. After that, we will prove Claim \ref{cla: no same support}, as it is quite involved.

\begin{lem}\label{lem: non neighbor 2 common}
Assume $S=\emptyset$ and $c\geq8$. Given a vertex $x$, if there is a vertex $y$ satisfying $|\operatorname{supp}(\mathbf{y})\cap\operatorname{supp}(\mathbf{x})|=2$, then for each vertex $z\in\Gamma_2(x)$, $|\operatorname{supp}(\mathbf{z})\cap\operatorname{supp}(\mathbf{x})|=2$ holds. 
\end{lem}
\begin{proof}By Lemma \ref{lem: general common support}, we have $x\not\sim y$ and then $(\mathbf{x},\mathbf{y})=0$ follows. Without loss of generality, we may assume 
\begin{equation}\label{eq:x y non neighbor 2 common}
\mathbf{x}=\mathbf{e}_1+\mathbf{e}_2+\mathbf{e}_3,~\mathbf{y}=\mathbf{e}_1-\mathbf{e}_2+\mathbf{e}_4.
\end{equation}
Suppose $z\in\Gamma_2(x)$ and $|\operatorname{supp}(\mathbf{z})\cap\operatorname{supp}(\mathbf{x})|=0$. We have  $(\mathbf{z},\mathbf{e}_4)\in\{0,1\}$, as $(\mathbf{z},\mathbf{y})\in\{0,1\}$. 
Now we are going to give a contradiction via counting the number of the common neighbors of $x$ and $z$. A fact is that for each vertex $w\in\Gamma_1(x)\cap\Gamma_1(z)$, 
\begin{equation}\label{eq: common x and z}
|\operatorname{supp}(\mathbf{w})\cap\operatorname{supp}(\mathbf{x})|=1 \text{ and }	|\operatorname{supp}(\mathbf{w})\cap\operatorname{supp}(\mathbf{z})|=1
\end{equation}
hold by Lemma \ref{lem: neighbor 1 common} {\rm (i)}, as $S=\emptyset$. Depending on the value of $(\mathbf{z},\mathbf{e}_4)$, we have the following two cases  to deal with.
\begin{enumerate}
	\item If $(\mathbf{z},\mathbf{e}_4)=1$, then we may let 
	\begin{equation}\label{eq:z 1 non neighbor 2 common}
		\mathbf{z}=\mathbf{e}_4+\mathbf{e}_5+\mathbf{e}_6.
	\end{equation}
For each vertex $w\in\Gamma_1(x)\cap\Gamma_1(z)$, we obtain, by \eqref{eq: common x and z}, that 
\begin{equation}\label{eq:w 1 non neighbor 2 common}
	(\mathbf{w},\mathbf{e}_i+\mathbf{e}_j)=2 \text{ for some }\{i,j\}\in\{\{1,5\},\{1,6\},\{2,4\},\{3,4\},\{3,5\},\{3,6\}\},
\end{equation}
considering $(\mathbf{w},\mathbf{y})\in\{0,1\}$. Thus, \[|\Gamma_1(x)\cap\Gamma_1(z)|\leq|\{\{1,5\},\{1,6\},\{2,4\},\{3,4\},\{3,5\},\{3,6\}\}|=6\]
by Lemma \ref{lem: neighbor 1 common} {\rm (ii)}.
 	\item  If $(\mathbf{z},\mathbf{e}_4)=0$, then we may let
 \begin{equation}\label{eq:z 2 non neighbor 2 common}
 	\mathbf{z}=\mathbf{e}_5+\mathbf{e}_6+\mathbf{e}_7.
 \end{equation}
Similarly, we have, for each vertex $w\in\Gamma_1(x)\cap\Gamma_1(z)$, that 
 \begin{equation}\label{eq:w 1 non neighbor 2 common}
 	(\mathbf{w},\mathbf{e}_i+\mathbf{e}_j)=2 \text{ for some }\{i,j\}\in\{\{1,5\},\{1,6\},\{1,7\},\{2,4\},\{3,5\},\{3,6\},\{3,7\}\},
 \end{equation}
and hence
 \[|\Gamma_1(x)\cap\Gamma_1(z)|\leq|\{\{1,5\},\{1,6\},\{1,7\},\{2,4\},\{3,5\},\{3,6\},\{3,7\}\}|=7.\]
 \end{enumerate} 
This contradicts that $c\geq8$. Therefore $|\operatorname{supp}(\mathbf{z})\cap\operatorname{supp}(\mathbf{x})|=2$ by Lemma \ref{lem: general common support}. Our lemma is proved. 
\end{proof}

\begin{lem}\label{lem: non neighbor 0 common}
	Assume $S=\emptyset$ and $c\geq4$. For each vertex $x$, there exists a vertex $y\in\Gamma_2(x)$ with $|\operatorname{supp}(\mathbf{y})\cap\operatorname{supp}(\mathbf{x})|=0$. 
\end{lem}
\begin{proof}
Assume that the lemma does not hold. We may let $\mathbf{x}=\mathbf{e}_1+\mathbf{e}_2+\mathbf{e}_3$. From the assumption, we find that $y\in \Gamma_2(x)$ implies $(\mathbf{y},\mathbf{e}_{i_1}-\mathbf{e}_{i_2})=2$ for some $\{i_1,i_2\}\in\{\{1,2\},\{1,3\},\{2,3\}\}$. Note that there exists no another vertex $y'$ such that $(\mathbf{y'},-\mathbf{e}_{i_1}+\mathbf{e}_{i_2})=2$, since $(\mathbf{y'},\mathbf{y})$ should equal $1$ or $0$. This means 
\begin{equation}\label{eq: Gamma2 non neighbor 0 common}
|\Gamma_2(x)|\leq|\{\{1,2\},\{1,3\},\{2,3\}\}|=3 \text{ for every }x
\end{equation} 
by Lemma \ref{lem: neighbor 1 common} {\rm (ii)}, as $S=\emptyset$.

We claim that the diameter of $\Gamma$ is $2$. Otherwise there is a path $x_0\sim x_1\sim x_2\sim x_3$ with $x_i\in\Gamma_i(x_0)$ for $i=1,2,3$. This leads to a contradiction immediately, as $3\geq|\Gamma_2(x_0)|\geq|\Gamma_1(x_1)\cap\Gamma_1(x_3)|=c\geq4$.  

Let us look back on the set $\Gamma_2(x)$. Since the diameter of $\Gamma$ is $2$, for each $y\in \Gamma_2(x)$, we have $\Gamma_1(y)=(\Gamma_1(y)\cap\Gamma_1(x))\cup(\Gamma_1(y)\cap\Gamma_2(x))$. Note that $|\Gamma_1(y)\cap\Gamma_1(x)|=c\leq k-2$ implies $|\Gamma_1(y)\cap\Gamma_2(x)|\geq2$. By applying \eqref{eq: Gamma2 non neighbor 0 common}, we conclude  $|\Gamma_2(x)|=3$ and $\Gamma_2(x)=\{y\}\cup(\Gamma_1(y)\cap\Gamma_2(x))$. In other words, the subgraph of $\Gamma$ induced on $\Gamma_2(x)$ is a $K_3$, as we choose the vertex $y$ in $\Gamma_2(x)$ arbitrarily. Now we may assume $\Gamma_2(x)=\{y,y_1,y_2\}$ and $\mathbf{y}=\mathbf{e}_1-\mathbf{e}_2+\mathbf{e}_4$. For $t=1,2$, since $y_t\sim y$, we have 
\begin{equation}\label{eq:1 non neighbor 0 common}
	|\operatorname{supp}(\mathbf{y_t})\cap\operatorname{supp}(\mathbf{y})|=1
\end{equation} by Lemma \ref{lem: neighbor 1 common} {\rm (i)}. Also, from the assumption, we have
\begin{equation}\label{eq:2 non neighbor 0 common}
 |\operatorname{supp}(\mathbf{y_t})\cap\operatorname{supp}(\mathbf{x})|=2.
\end{equation} 
Hence $(\mathbf{y_t},\mathbf{e}_4)=0$. More precisely, $(\mathbf{y_{t_1}},\mathbf{e}_1-\mathbf{e}_3)=2,~(\mathbf{y_{t_1}},\mathbf{e}_2)=0$ and $(\mathbf{y_{t_2}},-\mathbf{e}_2+\mathbf{e}_3)=2,~(\mathbf{y_{t_2}},\mathbf{e}_1)=0$, where $\{t_1,t_2\}=\{1,2\}$. But in this case, $(\mathbf{y_{t_1}},\mathbf{y_{t_2}})\leq 0$, which contradicts $(\mathbf{y_{t_1}},\mathbf{y_{t_2}})=1$ as $y_{t_1}\sim y_{t_2}$.  Therefore, the lemma holds.
\end{proof}

\begin{lem}\label{lem: conclusion}
Assume $S=\emptyset$ and $c\geq9$. For any two distinct vertices $x$ and $y$, the following hold:
\begin{enumerate}
\item $|\operatorname{supp}(\mathbf{x})\cap\operatorname{supp}(\mathbf{y})|=1$ holds if $x\sim y$, and 
$|\operatorname{supp}(\mathbf{x})\cap\operatorname{supp}(\mathbf{y})|=0$ holds if $x\not\sim y$.
\item If $d(x,y)=2$, then for any $i\in \operatorname{supp}(\mathbf{x})$ and $j\in \operatorname{supp}(\mathbf{y})$, there exists $w\in\Gamma_1(x)\cap\Gamma_1(y)$ such that $(\mathbf{w},\sigma_x(i)\mathbf{e}_{i}+\sigma_y(j)\mathbf{e}_{j})=2$ holds.
\item $c=9$.
\item The diameter of $\Gamma$ is $2$.
\item If $\operatorname{supp}(\mathbf{x})\cap\operatorname{supp}(\mathbf{y})=\{\ell\}$, then for any $i\in \operatorname{supp}(\mathbf{x})-\{\ell\}$ and $j\in \operatorname{supp}(\mathbf{y})-\{\ell\}$, there exists $w\in\Gamma_1(x)\cap\Gamma_1(y)$ such that $(\mathbf{w},\sigma_x(i)\mathbf{e}_{i}+\sigma_y(j)\mathbf{e}_{j})=2$ holds.
\end{enumerate}
\end{lem}
\begin{proof}
We may let $\mathbf{x}=a_1\mathbf{e}_{i_1}+a_2\mathbf{e}_{i_2}+a_3\mathbf{e}_{i_3}$ and  $\mathbf{y}=b_1\mathbf{e}_{j_1}+b_2\mathbf{e}_{j_2}+b_3\mathbf{e}_{j_3}$, where $a_1,a_2,a_3,b_1,b_2,b_3\in\{1,-1\}$.

{\rm (i)} If $x\sim y$, it follows from Lemma \ref{lem: neighbor 1 common} {\rm (i)}. If $x\not\sim y$, it follows from Lemma \ref{lem: general common support}, Lemma \ref{lem: non neighbor 2 common} and Lemma \ref{lem: non neighbor 0 common} as $c\geq9$.

{\rm (ii)}--{\rm (iii)} We have, by {\rm (i)}, that $i_1,i_2,i_3,j_1,j_2,j_3$ are pairwise different, and for each vertex $z\in\Gamma_1(x)\cap\Gamma_1(y)$,  $|\operatorname{supp}(\mathbf{z})\cap\operatorname{supp}(\mathbf{x})|=|\operatorname{supp}(\mathbf{z})\cap\operatorname{supp}(\mathbf{x})|=1$. Thus, there exist $i'\in\{i_1,i_2,i_3\}$ and $j'\in\{j_1,j_2,j_3\}$ such that $(\mathbf{z}, a_{x'}\mathbf{e}_{i'}+b_{j'}\mathbf{e}_{j'})=2$. By Lemma \ref{lem: neighbor 1 common} {\rm (ii)}, we have
\[
\begin{split}
	9\leq c&=|\Gamma_1(x)\cap\Gamma_1(y)|\\
	&=|\{z\mid (\mathbf{z}, a_{i'}\mathbf{e}_{i'}+b_{j'}\mathbf{e}_{j'})=2,i'\in\{i_1,i_2,i_3\},j'\in\{j_1,j_2,j_3\} \}|\\
	&\leq|\{\{i_1,j_1\},\{i_1,j_2\},\{i_1,j_3\},\{i_2,j_1\},\{i_2,j_2\},\{i_2,j_3\},\{i_3,j_1\},\{i_3,j_2\},\{i_3,j_3\}\}|\\
	&\leq  9.
\end{split}
\] 
Hence {\rm (ii)} and {\rm (iii)}  hold.

{\rm (iv)} Let $x'$ and $y'$ be two vertices of distance $2$. To prove {\rm (iv)}, it is sufficient to prove that $d(z,x')\leq 2$ for any vertex $z\in\Gamma_1(y')$. Without loss of generality, we may assume $z\not\sim x'$.
As $z\sim y'$, we have $|\operatorname{supp}(\mathbf{z})\cap\operatorname{supp}(\mathbf{y'})|=1$ by {\rm (i)}. From  {\rm (ii)}, we find that there must be a vertex $w'\in \Gamma_1(x')\cap\Gamma_1(y')$, such that $\operatorname{supp}(\mathbf{w'})\cap\operatorname{supp}(\mathbf{y})=\operatorname{supp}(\mathbf{z})\cap\operatorname{supp}(\mathbf{y})$. Thus, $|\operatorname{supp}(\mathbf{w'})\cap\operatorname{supp}(\mathbf{z})|=1$ and $w'\sim z$, which implies $d(z,x')=2$. 

{\rm (v)} We may assume $i=i_1$ and $j=j_1$. Let $z$ be a vertex in $\Gamma_2(y)$. Then $|\operatorname{supp}(\mathbf{z})\cap\operatorname{supp}(\mathbf{y})|=0$ follows. 
By {\rm (ii)}, we can find at least two vertices $w_1$ and $w_2$ in $\Gamma_1(y)\cap \Gamma_1(z)$ such that $(\mathbf{w_1},b_{j_1}\mathbf{e}_{j_1})=(\mathbf{w_2},b_{j_1}\mathbf{e}_{j_1})=1$ and  $(\mathbf{w_1},\mathbf{e}_{j_2})=(\mathbf{w_2},\mathbf{e}_{j_2})=(\mathbf{w_1},\mathbf{e}_{\ell})=(\mathbf{w_2},\mathbf{e}_{\ell})=0$. If either $(\mathbf{w_1},a_{i_1}\mathbf{e}_{i_1})$ or $(\mathbf{w_2},a_{i_1}\mathbf{e}_{i_1})$ equals $1$, then the existence of $w$ is shown. Otherwise $(\mathbf{w_1},\mathbf{e}_{i_1})=(\mathbf{w_2},\mathbf{e}_{i_1})=0$. Note that now we have either $(\mathbf{w_1},\mathbf{e}_{i_2})=0$ or $(\mathbf{w_2},\mathbf{e}_{i_2})=0$, as $|\operatorname{supp}(\mathbf{w_1})\cap\operatorname{supp}(\mathbf{w_2})|\leq1$ by {\rm (i)}. This means that at least one of $w_1$ and $w_2$ is not adjacent to $x$. Let us assume $w_1\not\sim x$. Since $d(x,w_1)=2$ by {\rm (iv)}, Using {\rm (ii)} again, we obtain the existence of $w$ in the set $\Gamma_1(x)\cap\Gamma_1(w_1)$. 
\end{proof}

Now we give the proof of Theorem \ref{c equals 9}.

\begin{proof}[\rm\textbf{Proof of Theorem \ref{c equals 9}}]
Since $S=\emptyset$ by Claim \ref{cla: no same support}, we may let $\mathcal{P}=\bigcup_{x\in V(\Gamma)}\operatorname{supp}(\mathbf{x})$ and $\mathcal{B}=\{\operatorname{supp}(\mathbf{x})\mid x\in V(\Gamma)\}$ a family of  $3$-element subsets of $\mathcal{P}$. Assume $i\in \operatorname{supp}(\mathbf{x})$ and $j\in \operatorname{supp}(\mathbf{y})$. If $x=y$, then $\{i,j\}$ is contained in the block $\operatorname{supp}(\mathbf{x})$. If $x\neq y$, Lemma \ref{lem: conclusion} {\rm (iv)},~{\rm (ii)} and {\rm (v)} guarantee that there exists $w\in\Gamma_1(x)\cap\Gamma_1(y)$ such that $\{i,j\}$ is contained in the block $\operatorname{supp}(\mathbf{w})$. The uniqueness of the block which contains $\{i,j\}$ is from Lemma \ref{lem: conclusion} {\rm (i)}.
Hence we conclude that $(\mathcal{P},\mathcal{B})$ is a Steiner triple system. Lemma \ref{lem: conclusion} {\rm (i)} says that the block graph of this Steiner triple system is exactly the graph $\Gamma$. Therefore, the theorem holds. 
\end{proof}

For the rest of this section, we will prove Claim \ref{cla: no same support}.


\begin{lem}\label{lem: 3 common no neighbor 0 common}
Assume $S\neq\emptyset$ and $c\geq7$. For any $x\in S$ and $y\in \Gamma_2(x)$,  $|\operatorname{supp}(\mathbf{x})\cap\operatorname{supp}(\mathbf{y})|=0$ holds.
\end{lem}
\begin{proof}
Suppose $|\operatorname{supp}(\mathbf{y})\cap\operatorname{supp}(\mathbf{x})|=2$. As $y\not\sim x$, $(\mathbf{x},\mathbf{y})=0$ follows. Without loss of generality, we may assume 
\[\mathbf{x}=\mathbf{e}_1+\mathbf{e}_2+\mathbf{e}_3,~\mathbf{y}=\mathbf{e}_1-\mathbf{e}_2+\mathbf{e}_4.\]
Let $x'$ be the mate of $x$. Considering $(\mathbf{x'},\mathbf{y})=1$ or $0$, we obtain
\[\mathbf{x'}=\mathbf{e}_1+\mathbf{e}_2-\mathbf{e}_3,\]
and thus $x'\not\sim y$. By Lemma \ref{lem: neighbor 1 common} {\rm (i)}, for any vertex $w\in \Gamma_1(x)-\{x'\}$,  \[|\operatorname{supp}(\mathbf{w})\cap\operatorname{supp}(\mathbf{x})|=1 \text{ and }(\mathbf{w},\mathbf{e}_{\operatorname{supp}(\mathbf{w})\cap\operatorname{supp}(\mathbf{x})})=1\]
hold. This means that for any vertex $w\in\Gamma_1(x)-\{x'\}$, $\operatorname{supp}(\mathbf{w})\cap\operatorname{supp}(\mathbf{x})\neq\{3\}$ as $(\mathbf{w},\mathbf{x'})\in\{0,1\}$, and furthermore, if $w\sim y$, then 
\begin{equation}\label{eq:1 3 common no neighbor 0 common}
(\mathbf{w},\mathbf{e}_1)=1 \text{ and } (\mathbf{w},\mathbf{e}_i)=0 \text{ for }i=2,3,4;	
\end{equation}
and
if $w\not\sim y$, then 
\begin{equation}\label{eq:2 3 common no neighbor 0 common}
(\mathbf{w},\mathbf{e}_1-\mathbf{e}_4)=2 \text{ or }(\mathbf{w},\mathbf{e}_2+\mathbf{e}_4)=2.	 
\end{equation}

Assume that there exists a vertex $w_1\in \Gamma_1(x)-\{x'\}-\Gamma_1(y)$ such that $(\mathbf{w_1},\mathbf{e}_1-\mathbf{e}_4)=2$. Let $\operatorname{supp}(\mathbf{w_1})=\{1,4,\ell_1\}$ for some $5\leq\ell_1\leq m$.  We find that in this case, for any vertex $u\in \Gamma_1(y)-\Gamma_1(x)$,
\[\operatorname{supp}(\mathbf{u})\cap\operatorname{supp}(\mathbf{x})=\emptyset \text { and }(\mathbf{u},\mathbf{e}_4+\sigma_{w_1}(\ell_1)\mathbf{e}_{\ell_1})=2.\]
Note that $2\leq k-c=|\Gamma_1(y)-\Gamma_1(x)|=|\{u\in V(\Gamma)\mid\operatorname{supp}(\mathbf{u})\cap\operatorname{supp}(\mathbf{x})=\emptyset \text { and } (\mathbf{u},\mathbf{e}_4+\sigma_{w_1}(\ell_1)\mathbf{e}_{\ell_1})=2\}|\leq2$ by Lemma \ref{lem: neighbor 1 common} {\rm (ii)}. Hence, $k-c=2$ and there exist two vertices $u_1,u_2\in \Gamma_1(y)$ such that
\[\mathbf{u_1}=\mathbf{e}_4+\sigma_{w_1}(\ell_1)\mathbf{e}_{\ell_1}+\mathbf{e}_{\ell_2},~\mathbf{u_2}=\mathbf{e}_4+\sigma_{w_1}(\ell_1)\mathbf{e}_{\ell_1}-\mathbf{e}_{\ell_2}\]
for some $\ell_2\not\in\{1,2,3,4,\ell_1\}$. From \eqref{eq:1 3 common no neighbor 0 common}, we easily obtain that for every $v\in \Gamma_1(x)\cap\Gamma_1(y)$, $\mathbf{v}$ satisfies $(\mathbf{v},\mathbf{e}_1)=1$, $(\mathbf{v},\mathbf{e}_i)=0$ for $i=2,3,4$. Since each of $(\mathbf{v},\mathbf{u_1})$, $(\mathbf{v},\mathbf{u_2})$ and $(\mathbf{v},\mathbf{w_1})$ equals $1$ or $0$, we have $(\mathbf{v},\mathbf{e}_{\ell_2})=0$ and $(\mathbf{v},\mathbf{e}_{\ell_1})=0$, and thus $v\not\sim u_1$. Notice that $d(v,u_1)=2$, as $v\sim y$ and $u_1\sim y$, but $c=|\Gamma_1(v)\cap\Gamma_1(u_1)|\leq |\Gamma_1(v)-\{x,x',w_1\}|=k-3$. This gives a contradiction. Therefore, for any $w\in\Gamma_1(x)-\{x'\}$ and $w\not\sim y$, $(\mathbf{w},\mathbf{e}_2+\mathbf{e}_4)=2$ holds by \eqref{eq:2 3 common no neighbor 0 common}, and 
\begin{equation}\label{eq: k-c geq 3}
k-c=|\Gamma_1(x)-\Gamma_1(y)|\leq |\{x'\}\cup\{w\in V(\Gamma)\mid (\mathbf{w},\mathbf{e}_2+\mathbf{e}_4)=2\}|\leq 1+2=3.
	\end{equation}  

Let $w_2\in \Gamma_1(x)-\Gamma_1(y)$ be a vertex such that $(\mathbf{w_2},\mathbf{e}_2+\mathbf{e}_4)=2$. Let $\operatorname{supp}(\mathbf{w_2})=\{2,4,\ell_3\}$ for some $5\leq\ell_3\leq m$. If $d(y,w_2)=2$, then we can easily obtain the following:
\begin{align*}
	k=|\Gamma_1(y)|&\geq|(\Gamma_1(y)\cap\Gamma_1(x))\cup(\Gamma_1(y)\cap\Gamma_1(w_2))|\\
	&=|\Gamma_1(y)\cap\Gamma_1(x)|+|\Gamma_1(y)\cap\Gamma_1(w_2)|-|\Gamma_1(y)\cap\Gamma_1(x)\cap\Gamma_1(w_2)| \\
	&\geq c+c-|\{w\in V(\Gamma)\mid (\mathbf{w},\mathbf{e}_1+\sigma_{w_2}(\ell_3)\mathbf{e}_{\ell_3})=2\}|\\
	&\geq 2c-2\\
	&\geq 2(k-3)-2.
\end{align*}
This means $k\leq 8$ and thus $c\leq 6$ as $k-c\geq2$, which contradicts the given condition $c\geq7$. So now we may assume $d(y,w_2)\geq3$ and look at the vertices in $\Gamma_1(y)-\Gamma_1(x)$. Note that for any vertex $z\in \Gamma_1(y)-\Gamma_1(x)$, 
 \[\operatorname{supp}(\mathbf{z})\cap\operatorname{supp}(\mathbf{x})=\emptyset \text { and }(\mathbf{z},\mathbf{e}_4-\sigma_{w_2}(\ell_3)\mathbf{e}_{\ell_3})=2,\]
 as $z\not\sim w_2$.
 Following a similar method as we used to obtain $u_1$ and $u_2$ in the previous paragraph, we can also find vertices $z_1$ and $z_2$ in $\Gamma_1(y)$ such that
 \[\mathbf{z_1}=\mathbf{e}_4-\sigma_{w_2}(\ell_3)\mathbf{e}_{\ell_3}+\mathbf{e}_{\ell_4},~\mathbf{z_2}=\mathbf{e}_4-\sigma_{w_2}(\ell_3)\mathbf{e}_{\ell_3}-\mathbf{e}_{\ell_4}\]
for some $\ell_4\not\in\{1,2,3,4,\ell_3\}$. Until now we can proceed a similar argument for $z_1,z_2, w_2, y$ as we did for $x,x',y,w_1$ in the previous paragraph (it is not hard to check $d(z_1,w_2)=2$), and obtain a contradiction. Therefore, our lemma holds.
 \end{proof} 

\begin{lem}\label{lem: 3 common no neighbor no mate}
Assume $S\neq\emptyset$ and $c\geq9$.  For any $x\in S$ and $y\in \Gamma_2(x)$, $y\not\in S$ holds.
\end{lem}
\begin{proof}
Suppose $y\in S$. By Lemma \ref{lem: 3 common no neighbor 0 common}, we have $|\operatorname{supp}(\mathbf{y})\cap\operatorname{supp}(\mathbf{x})|=0$. Thus we may assume  
$\mathbf{x}=\mathbf{e}_1+\mathbf{e}_2+\mathbf{e}_3$ and $\mathbf{y}=\mathbf{e}_4+\mathbf{e}_5+\mathbf{e}_6.$ Let $x'$ and $y'$ be the mates of $x$ and $y$ respectively. Without loss of generality, we may also assume $\mathbf{x'}=\mathbf{e}_1+\mathbf{e}_2-\mathbf{e}_3$ and $\mathbf{y'}=\mathbf{e}_4+\mathbf{e}_5-\mathbf{e}_6.$ Now it is easy to see 
\begin{align*}
	9\leq c&=|\Gamma_1(x)\cap\Gamma_1(y)|\\
	&\leq |\{w\in V(\Gamma)\mid (\mathbf{w},\mathbf{e}_i+\mathbf{e}_j)=2,i\in\{1,2\},j\in\{4,5\}\}\\
	&\leq 2|\{\{1,4\},\{1,5\},\{2,4\},\{2,5\}\}|=8
\end{align*}
by Lemma \ref{lem: neighbor 1 common} {\rm (ii)}.
This gives a contradition. Hence $y\not\in S$.
\end{proof}

\begin{lem}\label{lem:prepartion supp}
	Assume $S\neq\emptyset$ and $c\geq 9$. Let $x\in S$ and $y\in V(\Gamma)$ be two vertices such that $d(x,y)=2$, and let $x'$ be the mate of $x$. If $\operatorname{supp}(\mathbf{x})\in\{i_1,i_2,i_3\}$ and $\sigma_{x}(i_1)\sigma_{x'}(i_1)=\sigma_{x}(i_2)\sigma_{x'}(i_2)=1$,
	then the following hold:
	\begin{enumerate}
		\item $y\in V(\Gamma)-S$ and $\operatorname{supp}(\mathbf{x})\cap\operatorname{supp}(\mathbf{y})=\emptyset$.
		\item $|\Gamma_1(x)\cap \Gamma_1(y)\cap S|=6$.
		\item $\bigcap_{w\in \Gamma_1(x)\cap \Gamma_1(y)\cap S}\operatorname{supp}(\mathbf{w})\in\{i_1,i_2\}$ and $|\bigcap_{w\in \Gamma_1(x)\cap \Gamma_1(y)\cap S}\operatorname{supp}(\mathbf{w})|=1$.
		\item For $i\in\bigcap_{w\in \Gamma_1(x)\cap \Gamma_1(y)\cap S}\operatorname{supp}(\mathbf{w})$ and $j\in \operatorname{supp}(\mathbf{y})$, there exist  $z_1\in\Gamma_1(x)\cap \Gamma_1(y)\cap S$ and its mate $z'_1$ such that $(\mathbf{z_1},\sigma_{x}(i)\mathbf{e}_{i}+\sigma_{y}(j)\mathbf{e}_{j})=(\mathbf{z'_1},\sigma_{x}(i)\mathbf{e}_{i}+\sigma_{y}(j)\mathbf{e}_{j})=2$ hold.
		\item For $i\in \{i_1,i_2\}-\bigcap_{w\in \Gamma_1(x)\cap \Gamma_1(y)\cap S}\operatorname{supp}(\mathbf{w})$ and $j\in \operatorname{supp}(\mathbf{y})$, there exists $z_2\in\Gamma_1(x)\cap \Gamma_1(y)-S$ such that $(\mathbf{z_2},\sigma_{x}(i)\mathbf{e}_{i}+\sigma_{y}(j)\mathbf{e}_{j})=2$ holds.
		\item $c=9$.
	\end{enumerate}
	\end{lem}
\begin{proof}
{\rm (i)} This follows from Lemma \ref{lem: 3 common no neighbor no mate} and Lemma \ref{lem: 3 common no neighbor 0 common} immediately.

{\rm (ii)}--{\rm (vi)} We may assume $\mathbf{x}=\mathbf{e}_{i_1}+\mathbf{e}_{i_2}+\mathbf{e}_{i_3},~\mathbf{x'}=\mathbf{e}_{i_1}+\mathbf{e}_{i_2}-\mathbf{e}_{i_3},~\mathbf{y}=\mathbf{e}_{j_1}+\mathbf{e}_{j_2}+\mathbf{e}_{j_3}$, and
pay our attention to the set $\Gamma_1(x)\cap\Gamma_1(y)$. For any $w\in \Gamma_1(x)\cap\Gamma_1(y)$,
\begin{equation} \label{eq: the common neighbor}
	|\operatorname{supp}(\mathbf{w})\cap\operatorname{supp}(\mathbf{x})|=1,~ |\operatorname{supp}(\mathbf{w})\cap\operatorname{supp}(\mathbf{y})|=1
	\end{equation}
hold by Lemma \ref{lem: neighbor 1 common} {\rm (i)} and {\rm (i)}, and hence there exist $\ell_1\in\{i_1,i_2\}$ and  $\ell_2\in\{j_1,j_2,j_3\}$ such that $(\mathbf{w},\mathbf{e}_{\ell_1}+\mathbf{e}_{\ell_2})=2$. Let 
\begin{align*}
T_{1,i_1}&=\{\{i_1,\ell\}\mid (\mathbf{w},\mathbf{e}_{i_1}+\mathbf{e}_{\ell})=2,w\in \Gamma_1(x)\cap \Gamma_1(y)\cap S,\ell\in\{j_1,j_2,j_3\}\},\\
T_{1,i_2}&=\{\{i_2,\ell\}\mid (\mathbf{w},\mathbf{e}_{i_2}+\mathbf{e}_{\ell})=2,w\in \Gamma_1(x)\cap \Gamma_1(y)\cap S,\ell\in\{j_1,j_2,j_3\}\}, \text{ and }\\
T_2&=\{\{\ell_1,\ell_2\}\mid (\mathbf{w},\mathbf{e}_{\ell_1}+\mathbf{e}_{\ell_2})=2,w\in \Gamma_1(x)\cap \Gamma_1(y)-S,\ell_1\in\{i_1,i_2\},\ell_2\in\{j_1,j_2,j_3\}\}.	 
\end{align*}
Then $(T_{1,i_1}\cup T_{1,i_2})\cap T_2=\emptyset$ by Lemma \ref{lem: neighbor 1 common} {\rm (ii)}. Note that for each $w\in T_{1,i_2}\cup T_{1,i_2}$, its mate $w'$ is also in  $T_{1,i_2}\cup T_{1,i_2}$ by \eqref{eq: the common neighbor}. Hence we have 
\begin{equation}\label{eq: size T1 and T2}
	\begin{split}
	|T_{1,i_1}|+|T_{1,i_2}|+|T_2|\leq &~|\{\{\ell_1,\ell_2\}\mid \ell_1\in\{i_1,i_2\},\ell_2\in\{j_1,j_2,j_3\}\}|=6,\\
	2|T_{1,i_1}|+2|T_{1,i_2}|+|T_2|=&~|\Gamma_1(x)\cap\Gamma_1(y)|=c\geq 9
	\end{split}
\end{equation}
by using Lemma \ref{lem: neighbor 1 common} {\rm (ii)} again. It follows immediately that 
\begin{equation}\label{eq: size T1}
	|T_{1,i_1}|+|T_{1,i_2}|\geq3.
	\end{equation} If both $T_{1,i_1}$ and $T_{1,i_2}$ are not empty, then \eqref{eq: size T1} implies that there exist $\ell_3,\ell_4\in\{j_1,j_2,j_3\}$ such that $\ell_3\neq \ell_4$ and $\{i_1,\ell_3\}\in T_{1,i_1}$, $\{i_2,\ell_4\}\in T_{1,i_2}$.  Let $w_{i_1,\ell_3}$ and  $w_{i_2,\ell_4}$ be the vertices in $\Gamma_1(x)\cap \Gamma_1(y)\cap S$ such that $(\mathbf{w_{i_1,\ell_3}},\mathbf{e}_{i_1}+\mathbf{e}_{\ell_3})=2$ and $(\mathbf{w_{i_2,\ell_4}},\mathbf{e}_{i_2}+\mathbf{e}_{\ell_4})=2$. After looking at the integral vectors corresponding to the mates of $w_{i_1,\ell_3}$ and  $w_{i_2,\ell_4}$, we can easily obtain $|\operatorname{supp}(\mathbf{w_{i_1,\ell_3}})\cap\operatorname{supp}(\mathbf{w_{i_2,\ell_4}})|=0$ and thus $w_{i_1,\ell_3}\not\sim w_{i_2,\ell_4}$. Since both  $w_{i_1,\ell_3}$ and  $w_{i_2,\ell_4}$ are adjacent to $x$, $d(w_{i_1,\ell_3},w_{i_2,\ell_4})=2$ follows. This contradicts Lemma \ref{lem: 3 common no neighbor no mate}, as both $w_{i_1,\ell_3}$ and $w_{i_2,\ell_4}$ are in $S$ and have distance $2$. So now we may assume $T_{1,i_2}=\emptyset$.  Then \eqref{eq: size T1} forces $|T_{1,i_1}|=3$, and {\rm (iii)} and {\rm (iv)} hold.
Since $|\Gamma_1(x)\cap \Gamma_1(y)\cap S|=2|T_{1,i_1}|$, {\rm (ii)} also follows. By \eqref{eq: size T1 and T2} we obtain $|T_2|=3$ and $T_2=\{\{i_2,j_1\},\{i_2,j_2\},\{i_2,j_3\}\}$ immediately . Now it is obvious to see both {\rm (v)} and {\rm (vi)} hold.
\end{proof}

\begin{lem}\label{lem:clique}  
	Assume $S\neq\emptyset$ and $c\geq 9$. Then the subgraph of $\Gamma$ induced on $S$ is a clique.
	\end{lem}
\begin{proof}
	Suppose not. Then let $d=\min\{d(x,y)\mid x,y\in S\text{ and }d(x,y)\geq2\}$. Assume $x_0,x_d\in S$ and $d(x_0,x_d)=d$. There is a path $x_0\sim x_1\sim\cdots\sim x_d$ in $\Gamma$ with $x_i\in\Gamma_i(x_0)$ for $i=1,\ldots,d$. If $d\geq3$, then by Lemma \ref{lem:prepartion supp} {\rm (ii)}, there exists $z\in \Gamma_1(x_{d-2})\cap\Gamma_1(x_{d})$ such that $z\in S$. Note that $d(x_0,z)=d-1\geq2$. It follows that $d-1=d(x_0,z)\geq \min\{d(x,y)\mid x,y\in S\text{ and }d(x,y)\geq2\}=d$, which gives a contradiction. Thus $d=2$. But Lemma \ref{lem: 3 common no neighbor no mate} tells that this is impossible. The proof is completed.
	\end{proof}	

\begin{lem}
Assume $S\neq\emptyset$ and $c\geq9$. Then $|\bigcap_{x\in S}\operatorname{supp}(\mathbf{x})|=1$ holds. In particular, there exists $a\in\{1,-1\}$, such that for every $y\in S$, $(\mathbf{y},\mathbf{e}_{\bigcap_{x\in S}\operatorname{supp}(\mathbf{x})})=a$ holds.
\end{lem}
\begin{proof}
It is sufficient to prove $|\bigcap_{x\in S}\operatorname{supp}(\mathbf{x})|=1$.
From Lemma \ref{lem:prepartion supp} {\rm (ii)}, we have $|S|\geq 8$. Let $x_1,x_2,x_1',x_2'$ be four vertices in $S$, where $x'_i$ is the mate of $x_i$ for $i=1,2$. Note that by Lemma \ref{lem:clique}, the subgraph of $\Gamma$ induced on $S$ is a clique. Hence, we may assume  
\[\mathbf{x_1}=\mathbf{e}_{1}+\mathbf{e}_{2}+\mathbf{e}_{3},~\mathbf{x'_1}=\mathbf{e}_{1}+\mathbf{e}_{2}-\mathbf{e}_{3},~\mathbf{x_2}=\mathbf{e}_{1}+\mathbf{e}_{4}+\mathbf{e}_{5},~\mathbf{x'_2}=\mathbf{e}_{1}+\mathbf{e}_{4}-\mathbf{e}_{5}\] 
by Lemma \ref{lem: neighbor 1 common} {\rm (i)}. 

Suppose that this lemma does not hold. Then we may assume that there exists a vertex $x_3\in S$ such that $\operatorname{supp}(\mathbf{x_3})\cap\operatorname{supp}(\mathbf{x_1})\neq \operatorname{supp}(\mathbf{x_3})\cap\operatorname{supp}(\mathbf{x_2})$. Let $x'_3$ be the mate of $x_3$. It follows that $(\mathbf{x_3},\mathbf{e}_{2}+\mathbf{e}_{4})=(\mathbf{x'_3},\mathbf{e}_{2}+\mathbf{e}_{4})=2$. We may assume
\[\mathbf{x_3}=\mathbf{e}_{2}+\mathbf{e}_{4}+\mathbf{e}_{\ell},~\mathbf{x'_3}=\mathbf{e}_{2}+\mathbf{e}_{4}-\mathbf{e}_{\ell}\text{ for some }6\leq\ell\leq m.\]
As $|S|\geq8$, there must be a vertex $x_4$ in $S-\{x_1,x'_1,x_2,x'_2,x_3,x'_3\}$. But it is not possible to define $\mathbf{x_4}$ such that both $(\mathbf{x_4},\mathbf{x_i})=1$ and $(\mathbf{x_4},\mathbf{x'_i})=1$ hold for every $i\in\{1,2,3\}$. Therefore, our lemma is proved.
\end{proof}

From now on, if $S\neq\emptyset$, then we may always assume $S=\{x_1,x'_1,x_2,x'_2,\ldots,x_s,x'_s\}$, where $s\geq 4$ and for each $i=1,2,\ldots,s$, the vertices $x_i$ and $x'_i$ are mates, and 
\[\mathbf{x_i}=\mathbf{e}_{i}+\mathbf{e}_{s+i}+\mathbf{e}_m,~\mathbf{x'_i}=\mathbf{e}_{i}-\mathbf{e}_{s+i}+\mathbf{e}_m.\]

\begin{lem}\label{lem: distance between S and non S} 
	Assume $S\neq\emptyset$ and $c\geq9$. For any $x\in S$ and $y\in V(\Gamma)-S$, $d(x,y)\leq2$ holds.
\end{lem}
\begin{proof}
	Suppose $d=d(x,y)\geq3$. Then we can find a path $x\sim y_1\sim y_2\sim\cdots\sim y_d=y$ with $y_i\in \Gamma_i(x)$ for $i=1,\ldots,d$.  Since $y_i\not\sim x$ for $i\geq2$, we have $y_i\not\in S$ by Lemma \ref{lem:clique}. We may assume $y_1\in S$, otherwise we can find a vertex $y'_1\in\Gamma_1(x)\cap\Gamma_1(y_2)\cap S$ by Lemma \ref{lem:prepartion supp} {\rm (ii)}  and replace $y_1$ by $y'_1$. Now we look at the vertices $y_1$ and $y_3$. Using Lemma \ref{lem:prepartion supp} {\rm (ii)} again, we can find a vertex $y'_2$ in  $\Gamma_1(y_1)\cap\Gamma_1(y_3)\cap S$. Notice that $y'_2$ is also adjacent to $x,$ as $y'_2\in S$. This implies $d(y_3,x)\leq2$, which contradicts $y_3\in \Gamma_3(x)$. Therefore, our lemma holds.
\end{proof}

\begin{lem}\label{lem: distance 2}
	If $S\neq\emptyset$ and $c\geq9$, then for any $y\in V(\Gamma)-S$, there exists $x\in S$ such that $d(x,y)=2$ holds. 
\end{lem}
\begin{proof}
Suppose not, then for every $x\in S$, $y$ is adjacent to $x$ by Lemma \ref{lem: distance between S and non S}. It follows, for $i=1,\ldots,s$, that $|\operatorname{supp}(\mathbf{y})\cap\operatorname{supp}(\mathbf{x_i})|=1,~(\mathbf{y},\mathbf{e}_{\operatorname{supp}(\mathbf{y})\cap\operatorname{supp}(\mathbf{x_i})})=1$ and $|\operatorname{supp}(\mathbf{y})\cap\operatorname{supp}(\mathbf{x'_i})|=1,~(\mathbf{y},\mathbf{e}_{\operatorname{supp}(\mathbf{y})\cap\operatorname{supp}(\mathbf{x'_i})})=1$. Note that $|S|\geq8$ by Lemma \ref{lem:prepartion supp} {\rm (ii)}. Thus
$m\in\operatorname{supp}(\mathbf{y}),~(\mathbf{y},\mathbf{e}_m)=1$ and $ \operatorname{supp}(\mathbf{y})\cap \{1,\ldots,2s\}=\emptyset$. We may assume $\mathbf{y}=\mathbf{e}_m+\mathbf{e}_{j_1}+\mathbf{e}_{j_2}$, where $j_1,j_2\in\{2s+1,\ldots,m-1\}$. 

Let $z$ be a vertex in $\Gamma_2(x_1)$ with $\operatorname{supp}(\mathbf{z})=\{\ell_1,\ell_2,\ell_3\}$. Lemma \ref{lem:prepartion supp} {\rm (iv)} implies that $\ell_1,\ell_2,\ell_3\in\{1,2,\ldots,s\}$ and  $\mathbf{z}=\mathbf{e}_{\ell_1}+\mathbf{e}_{\ell_2}+\mathbf{e}_{\ell_3}$. Note that $y$ and $z$ are not adjacent, as $(\mathbf{y},\mathbf{z})=0$, and both are adjacent to $x_{\ell_1}$. So $d(y,z)=2$ follows and $|\Gamma_1(y)\cap\Gamma_1(z)|=9$ by Lemma \ref{lem:prepartion supp} {\rm (vi)}. Now let us look at the common neighbors of $y$ and $z$. For $w\in \Gamma_1(y)\cap\Gamma_1(z)$, either $(\mathbf{w},\mathbf{e}_{m}+\mathbf{e}_{\ell_p})=2$ or $(\mathbf{w},\mathbf{e}_{j_q}+\mathbf{e}_{\ell_p})=2$, where $q\in\{1,2\}$ and $p\in\{1,2,3\}$. Hence
	\[
	\begin{split}
		9=c&=|\Gamma_1(y)\cap\Gamma_1(z)|\\
		&=|\{w\mid (\mathbf{w}, \mathbf{e}_{m}+\mathbf{e}_{\ell_p})=2,~p\in\{1,2,3\}\}|+|\{w\mid (\mathbf{w},\mathbf{e}_{j_q}+\mathbf{e}_{\ell_p})=2,~q\in\{1,2\},~p\in\{1,2,3\}\}|\\
		&=|\{x_{\ell_1},x'_{\ell_1},x_{\ell_2},x'_{\ell_2},x_{\ell_3},x'_{\ell_3}\}|+|\{w\mid (\mathbf{w},\mathbf{e}_{j_q}+\mathbf{e}_{\ell_p})=2,~q\in\{1,2\},~p\in\{1,2,3\}\}|.
	\end{split}
	\] 
	This means that there must be a vertex $w'\in \Gamma_1(y)\cap\Gamma_1(z)$ such that
	$(\mathbf{w'},\mathbf{e}_{j_{q'}}+\mathbf{e}_{\ell_{p'}})=2$ for some $q'\in\{1,2\}$ and $p'\in\{1,2,3\}$.  For the convenience, we may let $q'=1$ and $p'=1$.
	
	We claim $(\mathbf{w'},\mathbf{e}_m)=0$. Otherwise $(\mathbf{w'},\mathbf{e}_{m})=1$, as $(\mathbf{w'},\mathbf{x}_{\ell_p})\in\{0,1\}$ and $(\mathbf{w'},\mathbf{x}_{\ell_p})\in\{0,1\}$ for $p=1,2,3$, and hence $(\mathbf{w'},\mathbf{e}_m+\mathbf{e}_{\ell_1})=(\mathbf{x_{j_1}},\mathbf{e}_m+\mathbf{e}_{\ell_1})=2$. This contradicts Lemma \ref{lem: neighbor 1 common} {\rm (ii)}, since  $w'$ is not the mate of $x_{\ell_1}$. 
	
	Now it is not hard to see $w'\not\sim x_{\ell_2}$ or $w'\not\sim x_{\ell_3}$, as $(\mathbf{w'},\mathbf{e}_{j_{1}}+\mathbf{e}_{\ell_{1}})=2$ and  $(\mathbf{w'},\mathbf{e}_m)=0$ imply either $(\mathbf{w'},\mathbf{x_{\ell_2}})=0$ or $(\mathbf{w'},\mathbf{x_{\ell_3}})=0$. Let us assume $w'\not\sim x_{\ell_2}$. Observe that both $w'$ and $x_{\ell_2}$ are adjacent to $y$. In other words, $w'\in \Gamma_2(x_{\ell_2})$. Lemma \ref{lem:prepartion supp} {\rm (iv)} forces $j_1\in\{1,\ldots,s\}$, and this gives a contradiction. We complete the proof.
\end{proof}

\begin{lem}\label{lem: support non S}
	Assume $c\geq9$ and $S\neq\emptyset$.  For every $y\in V(\Gamma)-S$, $\operatorname{supp}(\mathbf{y})$ is a $3$-element subset of $\{1,2,\ldots,s\}$. In particular, $y$ has exactly $6$ neighbors in $S$.
\end{lem}
\begin{proof}
By Lemma \ref{lem: distance 2}, there exists $x\in S$ such that $d(x,y)=2$ holds. Thus there exist $j_1,j_2,j_3\in \{1,2,\ldots,s\}$ such that $\operatorname{supp}(\mathbf{y})=\{j_1,j_2,j_3\}$ and $\mathbf{y}=\mathbf{e}_{j_1}+\mathbf{e}_{j_2}+\mathbf{e}_{j_3}$ by Lemma  \ref{lem:prepartion supp} {\rm (iv)}. Now it is not hard to see that $\Gamma_1(y)\cap S=\{x_{j_1},x'_{j_1},x_{j_2},x'_{j_2},x_{j_3},x'_{j_3}\}$. Therefore, our lemma holds.
\end{proof}

\begin{proof}[\rm\textbf{Proof of Claim \ref{cla: no same support}}]
Suppose that $S\neq\emptyset$. Let $\mathcal{P}'=\{1,2,\ldots,s\}$ and let $\mathcal{B}'=\{\operatorname{supp}(\mathbf{y})\mid y\in V(\Gamma)-S\}$. Then $\mathcal{B}'$ is a family of the $3$-element subsets of $\mathcal{P}'$ by Lemma \ref{lem: support non S}. 
 Now we claim that $(\mathcal{P}',\mathcal{B}')$ is a Steiner triple system $STS(s)$. 

Let $\{i_1,i_2\}$ be a $2$-element subset of $\mathcal{P}'$. We look at the sets $\Gamma_1(x_{i_1})-S$ and $\Gamma_1(x_{i_2})-S$. Since $S$ is clique and $\Gamma$ is a connected regular graph, there exists a positive integer $r$, such that, for every $x\in S$,
\begin{equation}\label{eq: r}
|\Gamma_1(x)-S|=r	
\end{equation}
holds. Thus both $\Gamma_1(x_{i_1})-S$ and $\Gamma_1(x_{i_2})-S$ are not empty. Let $y$ be a vertex in $\Gamma_1(x_{i_1})-S$. Then $(\mathbf{y},\mathbf{e}_{i_1})=1$ by Lemma \ref{lem: support non S}.  If $y\sim x_{i_2}$, then $(\mathbf{y},\mathbf{e}_{i_2})=1$ and $\{i_1,i_2\}$ is contained in the block $\operatorname{supp}(\mathbf{y})$. If $y\not\sim x_{i_2}$, then $d(y,x_{i_2})=2$ by Lemma \ref{lem: distance between S and non S}. From Lemma \ref{lem:prepartion supp} {\rm (v)}, there exists $z\in\Gamma_1(x_{i_2})\cap\Gamma_1(y)-S$ such that $i_1,i_2\in \operatorname{supp}(\mathbf{z})$, that is, $\{i_1,i_2\}$ is contained in the block $\operatorname{supp}(\mathbf{z})$. For the uniqueness of blocks containing  $\{i_1,i_2\}$, it is not hard to check. Therefore $(\mathcal{P}',\mathcal{B}')$ is a Steiner triple system $STS(s)$ and its block graph is the subgraph of $\Gamma$ induced on $V(\Gamma)-S$.

Let $\Gamma_1$ be the subgraph induced on $S$ and $\Gamma_2$  the subgraph induced on $V(\Gamma)-S$. Then $\Gamma_1$ is a clique with $2s$ vertices, and $\Gamma_2$ is a strongly regular graph with parameters $(\frac{s(s-1)}{6},\frac{3(s-3)}{2},\frac{s+3}{2}, 9)$ (see \cite[p.~5]{Cioaba.2014}). From \eqref{eq: r}, every vertex in $\Gamma_1$ has exactly $r$ neighbors in $\Gamma_2$, and from Lemma \ref{lem: support non S}, every vertex in $\Gamma_2$ has exactly $6$ neighbors in $\Gamma_2$. Thus
\begin{equation}\label{eq: value 1 of r}
r\cdot 2s=6\cdot \frac{s(s-1)}{6}.
\end{equation}
Considering that $\Gamma$ is regular, we have
\begin{equation}\label{eq: value 2 of r}
(2s-1)+r=\frac{3(s-3)}{2}+6.
\end{equation}
Solving \eqref{eq: value 1 of r} and \eqref{eq: value 2 of r}, we obtain $s=3$, which contradicts the fact $s\geq4$. 

The proof is completed.
\end{proof}

\section{Proof of Theorem \ref{1 integrability}}\label{proof the 1-integrability}
In this section, we will give the proof of Theorem \ref{1 integrability}. But before that, we need to introduce our main tool: Hoffman graphs.

\subsection{Hoffman graphs}
In this subsection, we will give the basic definitions and lemmas related to Hoffman graphs. For more details, we refer the readers to \cite{Koolen.2021a}.

\begin{de}A Hoffman graph $\mathfrak{h}$ is a pair $(H,\ell)$, where $H=(V(H),E(H))$ is a graph and $\ell (H):V \to\{f,s\}$ is a labeling map satisfying the following conditions:
	\begin{enumerate}
		\item vertices with label $f$ are pairwise non-adjacent,
		\item every vertex with label $f$ is adjacent to at least one vertex with label $s$.
	\end{enumerate}
\end{de}
We call a vertex with label $s$ a \emph{slim vertex}, and a vertex with label $f$ a \emph{fat vertex}. We denote by $V_{\mathrm{slim}}(\mathfrak{h})$ (resp.~$V_{\mathrm{fat}}(\mathfrak{h})$) the set of slim (resp.~fat) vertices of $\mathfrak{h}$.

\vspace{0.1cm}
For a vertex $x$ of $\mathfrak{h}$, we define $N_{\mathfrak{h}}^{\mathrm{slim}}(x)$ (resp.~$N_{\mathfrak{h}}^{\mathrm{fat}}(x)$) the set of slim (resp.~fat) neighbors of $x$ in $\mathfrak{h}$. If every slim vertex of $\mathfrak{h}$ has a fat neighbor, then we call $\mathfrak{h}$ \emph{fat}. 
In a similar fashion, we define $N^{\mathrm{slim}}_\mathfrak{h}(x_1,x_2)$ (resp.~$N^{\mathrm{fat}}_\mathfrak{h}(x_1,x_2)$) the set of common slim (resp.~fat) neighbors of $x_1$ and $x_2$ in $\mathfrak{h}$.

\vspace{0.1cm}
The \emph{slim graph} of $\mathfrak{h}:=(H,\ell)$ is the subgraph of $H$ induced on $V_{\mathrm{slim}}(\mathfrak{h})$. Note that any graph can be considered as a Hoffman graph with only slim vertices, and vice versa. We will not distinguish between Hoffman graphs with only slim vertices and graphs.

The \emph{special graph} of $\mathfrak{h}$ is the signed graph
	\[
	\mathcal{S}(\mathfrak{h}):=(V(\mathcal{S}(\mathfrak{h})),E^+(\mathcal{S}(\mathfrak{h})),E^-(\mathcal{S}(\mathfrak{h}))),
	\]
	where $V(\mathcal{S}(\mathfrak{h}))=V_{\mathrm{slim}}(\mathfrak{h})$ and
	\begin{align*}
		E^+(\mathcal{S}(\mathfrak{h}))&=\{\{x,y\}\mid x,y\in V_{\mathrm{slim}}(\mathfrak{h}),\{x,y\}\in E(\mathfrak{h}), N_\mathfrak{h}^{\mathrm{fat}}(x,y)=\emptyset\},\\
		E^-(\mathcal{S}(\mathfrak{h}))&=\{\{x,y\}\mid x,y\in V_{\mathrm{slim}}(\mathfrak{h}),\{x,y\}\in E(\mathfrak{h}), |N_\mathfrak{h}^{\mathrm{fat}}(x,y)|\geq2\}\\
		&\cup \{\{x,y\}\mid x,y\in V_{\mathrm{slim}}(\mathfrak{h}),x\neq y,\{x,y\}\not\in E(\mathfrak{h}), N_\mathfrak{h}^{\mathrm{fat}}(x,y)\neq\emptyset\}.
	\end{align*}

\vspace{0.1cm}
A Hoffman graph $\mathfrak{h}_1:= (H_1, \ell_1)$ is called an (\emph{proper}) \emph{induced Hoffman subgraph} of $\mathfrak{h}:=(H, \ell)$, if $H_1$ is an (proper) induced subgraph of $H$ and $\ell_1(x) = \ell(x)$ holds for every vertex $x$ of $H_1$.

\vspace{0.1cm}
Let $W$ be a subset of $V_{\mathrm{slim}}(\mathfrak{h})$. An induced Hoffman subgraph of $\mathfrak{h}$ generated by $W$, denoted by $\langle W\rangle_{\mathfrak{h}}$, is the Hoffman subgraph of $\mathfrak{h}$ induced on $W \cup\{f\in V_{\mathrm{fat}}(\mathfrak{h})\mid f \sim w \text{ for some }w\in W \}$.

\vspace{0.1cm}
For a fat vertex $f$ of $\mathfrak{h}$, a \emph{quasi-clique} (with respect to $f$) is a subgraph of the slim graph of $\mathfrak{h}$ induced on the slim vertices adjacent to $f$ in $\mathfrak{h}$, and we denote it by $Q_{\mathfrak{h}}(f)$.


\begin{de}\label{fatnbr}
Let $\mathfrak{h}$ be a Hoffman graph with $n$ slim vertices. The special matrix $Sp(\mathfrak{h})$ of $\mathfrak{h}$ is an $n\times n$ matrix indexed by the slim vertex set of $\mathfrak{h}$ such that for any two slim vertices $x$ and $y$ of $\mathfrak{h}$,
	\[
Sp(\mathfrak{h})_{x,y}=\left\{
\begin{array}{ll}
    -|N_\mathfrak{h}^{\mathrm{fat}}(x)| & \text{if } x=y, \\
	1-|N_\mathfrak{h}^{\mathrm{fat}}(x,y)| & \text{if } x\sim y, \\
	-|N_\mathfrak{h}^{\mathrm{fat}}(x,y)| & \text{otherwise}.
\end{array}
\right.
\]
\end{de}

The \emph{eigenvalues} of $\mathfrak{h}$ are the eigenvalues of its special matrix $Sp(\mathfrak{h})$, and the smallest eigenvalue of $\mathfrak{h}$ is denoted by  $\lambda_{\min}(\mathfrak{h})$. 

For the smallest eigenvalues of Hoffman graphs and their induced Hoffman subgraphs, Woo and Neumaier showed the following inequality.
\begin{lem}[{\cite[Corollary 3.3]{Woo}}]\label{hoff}
	If $\mathfrak{h}_1$ is an induced Hoffman subgraph of a Hoffman graph $\mathfrak{h}$, then $\lambda_{\min}(\mathfrak{h}_1)\geq\lambda_{\min}(\mathfrak{h})$ holds.
\end{lem}

\begin{de}\label{directsummatrix}
	Let $\mathfrak{h}^1$ and $\mathfrak{h}^2$ be two Hoffman graphs. A Hoffman graph $\mathfrak{h}$ is the sum of $\mathfrak{h}^1$ and $\mathfrak{h}^2$, denoted by $\mathfrak{h} =\mathfrak{h}^1\uplus\mathfrak{h}^2$, if $\mathfrak{h}$ satisfies the following condition:
	
	There exists a partition $\big\{V_{\mathrm{slim}}^1(\mathfrak{h}),V_{\mathrm{slim}}^2(\mathfrak{h})\big\}$ of $V_{\mathrm{slim}}(\mathfrak{h})$ such that induced Hoffman subgraphs generated by $V_{\mathrm{slim}}^i(\mathfrak{h})$ are $\mathfrak{h}^i$ for $i=1,2$ and
	\[Sp(\mathfrak{h})=
	\begin{pmatrix}
		Sp(\mathfrak{h}^1) & O \\
		O& Sp(\mathfrak{h}^2)
	\end{pmatrix}
	\] with respect to the partition $\big\{V_{\mathrm{slim}}^1(\mathfrak{h}),V_{\mathrm{slim}}^2(\mathfrak{h})\big\}$ of $V_{\mathrm{slim}}(\mathfrak{h})$.
\end{de}

We can check that $\mathfrak{h}$ is a sum of two non-empty Hoffman graphs if and only if its special matrix $Sp(\mathfrak{h})$ is a block matrix with at least two blocks. If $\mathfrak{h} =\mathfrak{h}^1 \uplus \mathfrak{h}^2$ for some non-empty Hoffman subgraphs $\mathfrak{h}^1$ and $\mathfrak{h}^2$, then we call $\mathfrak{h}$ \emph{decomposable} with $\{\mathfrak{h}^1,\mathfrak{h}^2\}$ as a \emph{decomposition} and call $\mathfrak{h}^1, \mathfrak{h}^2$ \emph{factors} of $\mathfrak{h}$. Otherwise, $\mathfrak{h}$ is called \emph{indecomposable}. Note that a Hoffman graph is indecomposable if and only if its special graph is connected.

There is an equivalent way to define the sum of two Hoffman graphs as follows.
\begin{lem}[{\cite[Lemma 2.11]{kyy}}]\label{combi}
	Let $\mathfrak{h}$ be a Hoffman graph and $\mathfrak{h}^1$ and $\mathfrak{h}^2$ be two induced Hoffman subgraphs of $\mathfrak{h}$. The Hoffman graph $\mathfrak{h}$ is the sum of $\mathfrak{h}^1$ and $\mathfrak{h}^2$ if and only if $\mathfrak{h}^1$, $\mathfrak{h}^2$, and $\mathfrak{h}$ satisfy the following conditions:
	\begin{enumerate}
		\item $V(\mathfrak{h})=V(\mathfrak{h}^1)\cup V(\mathfrak{h}^2);$
		\item $\big\{V_{\mathrm{slim}}(\mathfrak{h}^1),V_{\mathrm{slim}}(\mathfrak{h}^2)\big\}$ is a partition of $V_{\mathrm{slim}}(\mathfrak{h});$
		\item if $x \in V_{\mathrm{slim}}(\mathfrak{h}^i),~f \in V_{\mathrm{fat}}(\mathfrak{h})$ and $x\sim f$, then $f\in V_{\mathrm{fat}}(\mathfrak{h}^i);$
		\item if $x \in V_{\mathrm{slim}}(\mathfrak{h}^1)$ and $y \in V_{\mathrm{slim}}(\mathfrak{h}^2)$, then $x$ and $y$ have at most one common fat neighbor, and they have one if and only if they are adjacent.
	\end{enumerate}
\end{lem}

\begin{de}
For a Hoffman graph $\mathfrak{h}$ and a positive integer $m$, 
\begin{enumerate}
	\item if there is a mapping $\phi: V(\mathfrak{h}) \to \mathbb{Z}^m$ satisfying
	\[
	(\phi(x),\phi(y))=\left\{
	\begin{array}{ll}
		t & \text{if } x=y \text{ and } x,y\in V_{\mathrm{slim}}(\mathfrak{h}), \\
		1 & \text{if } x=y \text{ and } x,y\in V_{\mathrm{fat}}(\mathfrak{h}), \\
		1 & \text{if } x\sim y, \\
		0 & \text{otherwise},
	\end{array}
	\right.
	\]
	then $\mathfrak{h}$ has an \emph{integral representation} of norm $t$.
	\item if there is a mapping $\psi: V_{\mathrm{slim}}(\mathfrak{h}) \to \mathbb{Z}^m$ satisfying
	\[
	(\psi(x),\psi(y))=\left\{
	\begin{array}{ll}
		t-|N_\mathfrak{h}^{\mathrm{fat}}(x)| & \text{if } x=y, \\
		1-|N_\mathfrak{h}^{\mathrm{fat}}(x,y)| & \text{if } x\sim y, \\
		-|N_\mathfrak{h}^{\mathrm{fat}}(x,y)| & \text{otherwise},
	\end{array}
	\right.
	\]
	then $\mathfrak{h}$ has an \emph{integral reduced representation} of norm $t$.
\end{enumerate}
\end{de}

\begin{lem}[{cf.~\cite[Theorem $2.8$]{HJAT}}]\label{rela}
	For a Hoffman graph $\mathfrak{h}$, the following conditions are equivalent:
	\begin{enumerate}
		\item $\mathfrak{h}$ has an integral representation of norm $t$;
		\item $\mathfrak{h}$ has an integral reduced representation of norm $t$.
    \end{enumerate}
\end{lem}

\begin{thm}[{cf.~\cite[Theorem $3.7$]{HJAT}}]\label{HJAT 1-integrable}
	Given a fat indecomposable Hoffman graph with smallest eigenvalue at least $-3$, if one of its slim vertices has at least two fat neighbors, then this Hoffman graph has an integral reduced representation of norm $3$. 
\end{thm}

\subsection{Sesqui-regular graphs and Hoffman graphs}
In this subsection, we will use Hoffman graphs as our main tool to complete the proof of Theorem \ref{1 integrability}. Let $\Gamma$ be a connected sesqui-regular graph with parameters $(n,k,c)$, where $c\geq9$, and smallest eigenvalue $\lambda_{\min}(\Gamma)\in[-3,-2)$. To prove that $\Gamma$ is $1$-integrable when $k$ is sufficiently large, the only case needed to be considered is $c\leq k-2$ by Theorem \ref{c equals k or k-1}. Note that $c\leq k-2$ leads straightway to $n-k-1\geq 1+(k-c)\geq 3$.

A crucial result is as follows:

\begin{thm}[{cf.~\cite[Lemma 5.1]{kyy3} and \cite[Theorem 4.1]{Koolen.2021a}}]\label{Hoffman graph}
There exists a positive constant $K_4$ such that, for any connected sesqui-regular graph $\Gamma$ with parameters $(n,k,c)$, where $n-k-1>2$, and with smallest eigennvalue at least $-3$, if $k\geq K_4$, then there exists a fat Hoffman graph $\mathfrak{h}$ which satisfies the followsing:
\begin{enumerate}
	\item $\mathfrak{h}$ has $\Gamma$ as its slim graph,
	\item $\mathfrak{h}$ has smallest eigenvalue at least $-3$,
	\item every quasi-clique in $\mathfrak{h}$ is a clique.
	\end{enumerate} 	
\end{thm}

\begin{thm}\label{Hoffman graph integrable}
Let $\Gamma$ be a connected sesqui-regular graph with parameters $(n,k,c)$ and  smallest eigennvalue $\lambda_{\min}(\Gamma)\in[-3,-2)$. Let $\mathfrak{h}$ be a fat Hoffman graph which satisfies the conditions {\rm (i)}--{\rm (iii)} in Theorem \ref{Hoffman graph}. If $k\geq106$ and $k-75\geq c\geq 9$, then $\mathfrak{h}$ has an integral reduced representation of norm $3$.
\end{thm}
\begin{proof}
It is sufficient to prove that every indecomposable factor of $\mathfrak{h}$ has an integral reduced representation of norm $3$. Assume that $\mathfrak{g}$ is an indecomposable factor which has no integral reduced representation of norm $3$. Then $\mathfrak{g}$ is fat by Lemma \ref{combi} {\rm (iii)}, as $\mathfrak{h}$ is fat, and 
\begin{equation}\label{eq: g one fat}
|N_{\mathfrak{g}}^{\mathrm{fat}}(x)|=1\text{ for }x\in V_{\mathrm{slim}}(\mathfrak{g})
\end{equation}
by Theorem \ref{HJAT 1-integrable}. Let $\mathcal{S}(\mathfrak{g})$ be the special graph of $\mathfrak{g}$. Since every quasi-clique in $\mathfrak{h}$ is a clique, for any two non-adjacent slim vertices $x$ and $y$ in  $\mathfrak{h}$, $N_\mathfrak{h}^{\mathrm{fat}}(x,y)=\emptyset$ holds. In particular, if both $x$ and $y$ are in  $\mathfrak{g}$, then $N_\mathfrak{g}^{\mathrm{fat}}(x,y)=\emptyset$ also holds. Taking in consideration \eqref{eq: g one fat},  we have $E^-(\mathcal{S}(\mathfrak{g}))=\emptyset$. Hence, we can regard $\mathcal{S}(\mathfrak{g})$ as a graph. The graph $\mathcal{S}(\mathfrak{g})$ is connected, since $\mathfrak{g}$ is indecomposable. 

Notice that $Sp(\mathfrak{g})=A(\mathcal{S}(\mathfrak{g}))-I$, where $A(\mathcal{S}(\mathfrak{g}))$ is the adjacency matrix of $\mathcal{S}(\mathfrak{g})$. Let $\lambda_{\min}(\mathcal{S}(\mathfrak{g}))$ be the smallest eigenvalue of $\mathcal{S}(\mathfrak{g})$. From Lemma \ref{hoff}, we have $\lambda_{\min}(\mathcal{S}(\mathfrak{g}))-1=\lambda_{\min}(\mathfrak{g})\geq\lambda_{\min}(\mathfrak{h})\geq-3$, that is, $\lambda_{\min}(\mathcal{S}(\mathfrak{g}))\geq-2$. Also, $\mathcal{S}(\mathfrak{g})$ is not a complete graph, as $\mathfrak{g}$ has no integral reduced representation of norm $3$.  So
\begin{equation}\label{eq eigenvalue of special graph}
	-2\leq \lambda_{\min}(\mathcal{S}(\mathfrak{g}))<-1.
\end{equation}
Using the fact that $\mathfrak{g}$ has no integral reduced representation of norm $3$ again, we conclude that $\mathcal{S}(\mathfrak{g})$ is not $1$-integrable. By Theorem \ref{thm:cameron}, we obtain
 \begin{equation}\label{eq slim vertex number g}
|V_{\mathrm{slim}}(\mathfrak{g})|=|V(\mathcal{S}(\mathfrak{g}))|\leq 36.
\end{equation}

\begin{cla}\label{cla: Dn like}
 Let $x$ and $y$ be two adjacent vertices in $\mathcal{S}(\mathfrak{g})$. There exist two vertices $y_1,y_2~(\neq y)$ in $\mathcal{S}(\mathfrak{g})$ such that both of them are adjacent to $x$ in  $\mathcal{S}(\mathfrak{g})$ and satisfy $N_{\mathfrak{g}}^{\mathrm{fat}}(y_1)=N_{\mathfrak{g}}^{\mathrm{fat}}(y_2)=N_{\mathfrak{g}}^{\mathrm{fat}}(y)$ in $\mathfrak{g}$.  
\end{cla}

Before showing Claim \ref{cla: Dn like}, we finish the proof of Theorem \ref{Hoffman graph integrable} by using this claim. 
Since the quasi-clique with respect to a fixed fat vertex in $\mathfrak{g}$ is a clique and every slim vertex in  $\mathfrak{g}$  has exactly one fat neighbor in $\mathfrak{g}$ by \eqref{eq: g one fat}, we observe that for any two vertices $z_1,z_2$ in $\mathcal{S}(\mathfrak{g})$, 
\begin{equation}\label{eq non adjacent}
	z_1\not\sim z_2 \text{ in }\mathcal{S}(\mathfrak{g}) \text{ if }N_{\mathfrak{g}}^{\mathrm{fat}}(z_1)=N_{\mathfrak{g}}^{\mathrm{fat}}(z_2) \text{ in }\mathfrak{g}.
\end{equation}

Now under the assumption that Claim \ref{cla: Dn like} holds, we
easily figure out that $\mathcal{S}(\mathfrak{g})$ contains an induced subgraph with smallest eigenvalue less than $-2$ as follows. Since $\mathcal{S}(\mathfrak{g})$ is connected, there are two adjacent vertices $x$ and $y$ in $\mathcal{S}(\mathfrak{g})$. By Claim \ref{cla: Dn like}, there exist two more vertices $y_1$ and $y_2$ which are adjacent to $x$ in $\mathcal{S}(\mathfrak{g})$ and have the same fat neighbor as $y$ in $\mathfrak{g}$. By \eqref{eq non adjacent}, the three vertices $y,y_1,y_2$ are pairwise non-adjacent in $\mathcal{S}(\mathfrak{g})$. Considering that the positions of $x$ and $y$ are symmetric. Using Claim \ref{cla: Dn like} again, we can also find two more vertices $x_1$ and $x_2$ which are adjacent to $y$ in $\mathcal{S}(\mathfrak{g})$ and satisfy $N_{\mathfrak{g}}^{\mathrm{fat}}(x_1)=N_{\mathfrak{g}}^{\mathrm{fat}}(x_2)=N_{\mathfrak{g}}^{\mathrm{fat}}(x)$ in $\mathfrak{g}$.  Similarly, we have that $x,x_1,x_2$ are pairwise non-adjacent in $\mathcal{S}(\mathfrak{g})$ by \eqref{eq non adjacent}.
\begin{enumerate}
	\item If either $y_1$ or $y_2$ is adjacent to one of $x_1$ and $x_2$, then $\mathcal{S}(\mathfrak{g})$ will contain one of the graphs in Figure \ref{fig:1} 
	\begin{figure}[H]
		\centering
		\begin{tikzpicture}[thick][node distance=1cm,on grid]
			\node[circle,inner sep=0.64mm,draw=black,fill=black](x)[label=left:$x$]{};
			\node[circle,inner sep=0.64mm,draw=black,fill=black](y)[right=of x,xshift=-0.4cm,yshift=0.5cm,label=above:$y$]{};
			\node[circle,inner sep=0.64mm,draw=black,fill=black](y1)[right=of x,xshift=-0.4cm,yshift=-0.5cm, label=below:$y_1~\text{or}~y_2$]{};
			\node[circle,inner sep=0.64mm,draw=black,fill=black](x1)[right=of y,xshift=-0.2cm,label=right:$x_i\text{,}~i\in\{1\text{,}2\}$]{};
			\node[circle,inner sep=0.64mm,draw=black,fill=black](x2)[right=of y1,xshift=-0.2cm,label=right:$x_{3-i}$]{};
			\draw[-](x)to(y);
			\draw[-](x)to(y1);
			\draw[-](y)to(x1);
			\draw[-](y)to(x2);
			\draw[-](y1)to(x1);
		\end{tikzpicture}\hspace{1cm}
		\begin{tikzpicture}[thick][node distance=1cm,on grid]
			\node[circle,inner sep=0.64mm,draw=black,fill=black](x)[label=left:$x$]{};
			\node[circle,inner sep=0.64mm,draw=black,fill=black](y)[right=of x,xshift=-0.4cm,yshift=0.5cm,label=above:$y$]{};
			\node[circle,inner sep=0.64mm,draw=black,fill=black](y1)[right=of x,xshift=-0.4cm,yshift=-0.5cm, label=below:$y_1~\text{or}~y_2$]{};
			\node[circle,inner sep=0.64mm,draw=black,fill=black](x1)[right=of y,xshift=-0.2cm,label=right:$x_1$]{};
			\node[circle,inner sep=0.64mm,draw=black,fill=black](x2)[right=of y1,xshift=-0.2cm,label=right:$x_2$]{};
			\draw[-](x)to(y);
			\draw[-](x)to(y1);
			\draw[-](y)to(x1);
			\draw[-](y)to(x2);
			\draw[-](y1)to(x1);
			\draw[-](y1)to(x2);
		\end{tikzpicture}
	    {\caption{Two graphs with smallest eigenvalue less than $-2$}\label{fig:1}}
	\end{figure}
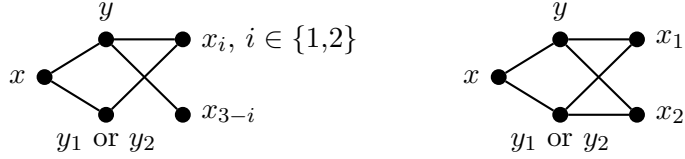
\vspace{-0.5cm}

	which have smallest eigenvalue less than $-2$ as an induced subgraph. 
	\item If both $y_1$ and $y_2$ are adjacent to neither $x_1$ nor $x_2$, then applying Claim \ref{cla: Dn like} to the adjacent vertices $x$ and $y_1$ again, we will find that there are two more new vertices $x_{1,1}$ and $x_{1,2}$ which are adjacent to $y_1$ in $\mathcal{S}(\mathfrak{g})$ and satisfy $N_{\mathfrak{g}}^{\mathrm{fat}}(x_{1,1})=N_{\mathfrak{g}}^{\mathrm{fat}}(x_{1,2})=N_{\mathfrak{g}}^{\mathrm{fat}}(x)$ in $\mathfrak{g}$. Note that $x,x_1,x_2,x_{1,1},x_{1,2}$ have the same fat neigbhor in $\mathfrak{g}$. This means that these five vertices are pairwise non-adjacent in $\mathcal{S}(\mathfrak{g})$ by \ref{eq non adjacent}. Now
if $y_2$ is adjacent to one of $x_{1,1}$ and $x_{1,2}$, then $\mathcal{S}(\mathfrak{g})$ contains an induced subgraph isomorphic to one of the graphs in Figure \ref{fig:1}, which has smallest eigenvalue less than $-2$. If $y_2$ is adjacent to neither $x_{1,1}$ nor $x_{1,2}$, then $\mathcal{S}(\mathfrak{g})$  contains the graph in Figure \ref{fig:2}
	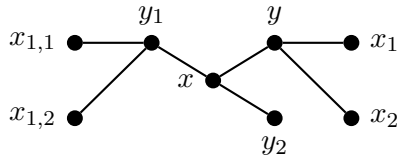
\begin{figure}[H]
		\centering
		\begin{tikzpicture}[thick][node distance=1cm,on grid]
			\node[circle,inner sep=0.64mm,draw=black,fill=black](x)[label=left:$x$]{};
			\node[circle,inner sep=0.64mm,draw=black,fill=black](y)[right=of x,xshift=-0.4cm,yshift=0.5cm,label=above:$y$]{};
			\node[circle,inner sep=0.64mm,draw=black,fill=black](y2)[right=of x,xshift=-0.4cm,yshift=-0.5cm,label=below:$y_2$]{};
			\node[circle,inner sep=0.64mm,draw=black,fill=black](x1)[right=of y,xshift=-0.2cm,label=right:$x_1$]{};
			\node[circle,inner sep=0.64mm,draw=black,fill=black](x2)[right=of y1,xshift=-0.2cm,label=right:$x_2$]{};
			\node[circle,inner sep=0.64mm,draw=black,fill=black](y1)[left=of x,xshift=0.4cm,yshift=0.5cm,label=above:$y_1$]{};
			\node[circle,inner sep=0.64mm](y3)[left=of x,xshift=0.4cm,yshift=-0.5cm]{};
			\node[circle,inner sep=0.64mm,draw=black,fill=black](x11)[left=of y1,xshift=0.2cm,label=left:$x_{1,1}$]{};
			\node[circle,inner sep=0.64mm,draw=black,fill=black](x12)[left=of y3,xshift=0.2cm,label=left:$x_{1,2}$]{};
			\draw[-](x)to(y);
			\draw[-](x)to(y1);
			\draw[-](x)to(y2);
			\draw[-](y)to(x1);
			\draw[-](y)to(x2);
			\draw[-](y1)to(x11);
			\draw[-](y1)to(x12);
		\end{tikzpicture}
	{\caption{A graph with smallest eigenvalue less than $-2$}\label{fig:2}}
	\end{figure}
\vspace{-0.5cm}

	with smalles eigenvalue less than $-2$ as an induced subgraph. 
\end{enumerate} 
This contradicts Lemma \ref{lem: interlacing} {\rm (ii)}, as $\mathcal{S}(\mathfrak{g})$ has smallest eigenvalue at least $-2$ by \eqref{eq eigenvalue of special graph}.

Hence every indecomposable factor of $\mathfrak{h}$ has an integral reduced representation of norm $3$. This completes the proof.
\end{proof}

\begin{proof}[\rm\textbf{Proof of Claim \ref{cla: Dn like}}]
Assume $N_{\mathfrak{g}}^{\mathrm{fat}}(x)={f_1}$ and $N_{\mathfrak{g}}^{\mathrm{fat}}(y)={f_2}$. Let $F=\bigcup_{w\in N_{\mathfrak{g}}^{\mathrm{slim}}(x)}N_{\mathfrak{g}}^{\mathrm{fat}}(w)-\{f_2\}$. Considering  $|V_{\mathrm{slim}}(\mathfrak{g})|\leq 36$ by \eqref{eq slim vertex number g}, we have
	\begin{equation}\label{size F}
		|F|\leq|\{w\mid w\in N_{\mathfrak{g}}^{\mathrm{slim}}(x)\}|-1\leq |V_{\mathrm{slim}}(\mathfrak{g})-\{x\}|-1\leq 34
	\end{equation}
	by \eqref{eq: g one fat} immediately. 
	
	Let $z_1,z_2,z_3\in V_{\mathrm{slim}}(\mathfrak{h})$ be three vertices. If $|N_{\mathfrak{h}}^{\mathrm{fat}}(z_1)\cap N_{\mathfrak{h}}^{\mathrm{fat}}(z_2)\cap N_{\mathfrak{h}}^{\mathrm{fat}}(z_3)|\geq2$, then $z_1,z_2$ and $z_3$ are pairwise adjacent in $\mathfrak{h}$, as they are in the quasi-clique with respect to a fat vertex in $N_{\mathfrak{h}}^{\mathrm{fat}}(z_1)\cap N_{\mathfrak{h}}^{\mathrm{fat}}(z_2)\cap N_{\mathfrak{h}}^{\mathrm{fat}}(z_3)$, which is actually a clique. Thus, the following Hoffman graph
	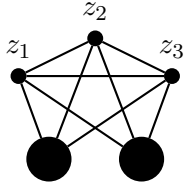
\begin{figure}[H]
		\centering
		\begin{tikzpicture}[thick][node distance=1cm,on grid]
			\node[circle,inner sep=0.64mm,draw=black,fill=black](x)[label=above:$z_1$]{};
			\node[circle,inner sep=0.64mm,draw=black,fill=black](y)[right=of x,xshift=-0.2cm,yshift=0.5cm,label=above:$z_2$]{};
			\node[circle,inner sep=0.64mm,draw=black,fill=black](z)[right=of y,xshift=-0.2cm,yshift=-0.5cm,label=above:$z_3$]{};
			\node[circle,inner sep=2mm,draw=black,fill=black](b1)[below=of x,xshift=0.4cm,yshift=0.3cm]{};
			\node[circle,inner sep=2mm,draw=black,fill=black](b2)[below=of z,xshift=-0.4cm,yshift=0.3cm]{};
			\draw[-](x)to(z);
			\draw[-](x)to(b1);
			\draw[-](x)to(b2);
			\draw[-](y)to(b1);
			\draw[-](y)to(b2);
			\draw[-](z)to(b1);
			\draw[-](z)to(b2);
			\draw[-](x)to(y);
			\draw[-](z)to(y);
		\end{tikzpicture}
	 {\caption{A Hoffman graph with smallest eigenvalue $-4$}\label{fig:}}
	\end{figure}
	\noindent with smallest eigenvalue $-4$ is an induced Hoffman subgraph of $\mathfrak{h}$, which contradicts $\lambda_{\min}(\mathfrak{h})\geq-3$ by Lemma \ref{hoff}. This leads us to obtain the following fact:
	
	\vspace{0.2cm}\noindent{\bf Fact 1}. For any two fat vertices $f_3$ and $f_4$ in $ V_{\mathrm{fat}}(\mathfrak{h})$,  $|N_{\mathfrak{h}}^{\mathrm{slim}}(f_3,f_4)|\leq2$ holds.
	
	\vspace{0.2cm}From Fact $1$ and \eqref{size F}, we obtain 
	\begin{equation}\label{eq: size F}
		|\bigcup\nolimits_{f\in F\cup\{f_1\}}N_{\mathfrak{h}}^{\mathrm{slim}}(f_2,f)|\leq 2(|F|+1)\leq 70.
	\end{equation}
	There is one more essential fact following from Lemma \ref{combi} {\rm (iv)} and \eqref{eq: g one fat}:
	
	\vspace{0.2cm}\noindent{\bf Fact 2}. Let $w_1\in V_{\mathrm{slim}}(\mathfrak{g})$ and $w_2\in V_{\mathrm{slim}}(\mathfrak{h})-V_{\mathrm{slim}}(\mathfrak{g})$ be two slim vertices and let $N_{\mathfrak{g}}^{\mathrm{fat}}(w_1)=\{f\}$. Then $w_1$ and $w_2$ are adjacent in $\mathfrak{h}$ if and only if $w_2\in N_{\mathfrak{h}}^{\mathrm{slim}}(f)$.
	
\vspace{0.2cm} We look at the set $N_{\mathfrak{h}}^{\mathrm{slim}}(y)$. By Fact $2$, a vertex $w_2\in V_{\mathrm{slim}}(\mathfrak{h})-V_{\mathrm{slim}}(\mathfrak{g})$ is in the set $N_{\mathfrak{h}}^{\mathrm{slim}}(y)$ if and only if  $w_2\in N_{\mathfrak{h}}^{\mathrm{slim}}(f_2)$. Then 
$k =|N_{\mathfrak{h}}^{\mathrm{slim}}(y)|
		=|N_{\mathfrak{g}}^{\mathrm{slim}}(y)|+|N_{\mathfrak{h}}^{\mathrm{slim}}(y)-N_{\mathfrak{g}}^{\mathrm{slim}}(y)|
	=|N_{\mathfrak{g}}^{\mathrm{slim}}(y)|+|\bigcup\nolimits_{f\in F\cup\{f_1\}}N_{\mathfrak{h}}^{\mathrm{slim}}(f_2,f)|+|N_{\mathfrak{h}}^{\mathrm{slim}}(f_2)-N_{\mathfrak{g}}^{\mathrm{slim}}(f_2)-\bigcup\nolimits_{f\in F\cup\{f_1\}}N_{\mathfrak{h}}^{\mathrm{slim}}(f_2,f)|
	\geq 106.$ From \eqref{eq slim vertex number g} and \eqref{eq: size F}, we easily obtain  $|N_{\mathfrak{g}}^{\mathrm{slim}}(y)|\leq 35$ and $|\bigcup\nolimits_{f\in F\cup\{f_1\}}N_{\mathfrak{h}}^{\mathrm{slim}}(f_2,f)|\leq 70$.  Thus, $|N_{\mathfrak{h}}^{\mathrm{slim}}(f_2)-N_{\mathfrak{g}}^{\mathrm{slim}}(f_2)-\bigcup\nolimits_{f\in F\cup\{f_1\}}N_{\mathfrak{h}}^{\mathrm{slim}}(f_2,f)|\geq1$, and we are able to find a vertex $z$ in $N_{\mathfrak{h}}^{\mathrm{slim}}(f_2)-N_{\mathfrak{g}}^{\mathrm{slim}}(f_2)-\bigcup\nolimits_{f\in F\cup\{f_1\}}N_{\mathfrak{h}}^{\mathrm{slim}}(f_2,f)$. 
	As $z\not\in N_{\mathfrak{h}}^{\mathrm{slim}}(f_1)$, we have $z\not\in N_{\mathfrak{h}}^{\mathrm{slim}}(x)$ by Fact $2$. Notice that both $x$ and $z$ are adjacent to $y$.  The distance between $x$ and $z$ in the slim graph $\Gamma$ of $\mathfrak{h}$ is $2$, and thus, 
	\begin{equation}\label{eq size of N(x,z)}
		9\leq c=|N_{\mathfrak{h}}^{\mathrm{slim}}(x,z)|=|(V_{\mathrm{slim}}(\mathfrak{h})-V_{\mathrm{slim}}(\mathfrak{g}))\cap N_{\mathfrak{h}}^{\mathrm{slim}}(x,z)|+|V_{\mathrm{slim}}(\mathfrak{g})\cap N_{\mathfrak{h}}^{\mathrm{slim}}(x,z)|. 
	\end{equation} 

  By Claim $2$, we have $(V_{\mathrm{slim}}(\mathfrak{h})-V_{\mathrm{slim}}(\mathfrak{g}))\cap N_{\mathfrak{h}}^{\mathrm{slim}}(x,z)=(V_{\mathrm{slim}}(\mathfrak{h})-V_{\mathrm{slim}}(\mathfrak{g}))\cap N_{\mathfrak{h}}^{\mathrm{slim}}(z,f_1)$. For any $w_1\in V_{\mathrm{slim}}(\mathfrak{g})$, if $w_1\in N_{\mathfrak{h}}^{\mathrm{slim}}(f_1)$, then $N_{\mathfrak{h}}^{\mathrm{fat}}(w_1)=\{f_1\}$ by \eqref{eq: g one fat}. Furthermore, $w_1$ and $z$ are not adjacent in $\mathfrak{h}$, as $z\not\in N_{\mathfrak{h}}^{\mathrm{slim}}(f_1)$. In other words, $w_1\not\in N_{\mathfrak{h}}^{\mathrm{slim}}(z)$. This means $	(V_{\mathrm{slim}}(\mathfrak{h})-V_{\mathrm{slim}}(\mathfrak{g}))\cap N_{\mathfrak{h}}^{\mathrm{slim}}(z,f_1)= N_{\mathfrak{h}}^{\mathrm{slim}}(z,f_1).$

Let $W=N_{\mathfrak{h}}^{\mathrm{slim}}(z,f_1)$. We will show $|W|\leq 6$. Since the quasi-clique with respect to the fat vertex $f_1$ in $\mathfrak{h}$ is a clique, we have $N_{\mathfrak{g}}^{\mathrm{slim}}(f_1)-\{x\}\subset N_{\mathfrak{h}}^{\mathrm{slim}}(x)$. Also by Fact 2, we have $N_{\mathfrak{h}}^{\mathrm{slim}}(f_1)-N_{\mathfrak{g}}^{\mathrm{slim}}(f_1)=N_{\mathfrak{h}}^{\mathrm{slim}}(x)-N_{\mathfrak{g}}^{\mathrm{slim}}(x)$. It follows that 
	\[
	\begin{split}
		|N_{\mathfrak{h}}^{\mathrm{slim}}(f_1)|&=|(N_{\mathfrak{h}}^{\mathrm{slim}}(f_1)-N_{\mathfrak{g}}^{\mathrm{slim}}(f_1))\cup N_{\mathfrak{g}}^{\mathrm{slim}}(f_1)|\\
		&=|(N_{\mathfrak{h}}^{\mathrm{slim}}(x)-N_{\mathfrak{g}}^{\mathrm{slim}}(x))\cup N_{\mathfrak{g}}^{\mathrm{slim}}(f_1)|\\
		&=|(N_{\mathfrak{h}}^{\mathrm{slim}}(x)\cup N_{\mathfrak{g}}^{\mathrm{slim}}(f_1))-(N_{\mathfrak{g}}^{\mathrm{slim}}(x)- N_{\mathfrak{g}}^{\mathrm{slim}}(f_1))|\\
		&\geq |N_{\mathfrak{h}}^{\mathrm{slim}}(x)\cup \{x\}|-|N_{\mathfrak{g}}^{\mathrm{slim}}(x)|\\
		&\geq k+1-35=k-34\geq c+41.
	\end{split}
	\]
As $|W|\leq |N_\mathfrak{h}^{\mathrm{slim}}(x,z)|\leq c$, there exists $W'\subset N_{\mathfrak{h}}^{\mathrm{slim}}(f_1)-W$ such that $|W'|=41$. Using the fact that the quasi-clique with respect to $f_1$ in $\mathfrak{h}$ is a clique again, we obtain that the quotient matrix of the subgraph of $\Gamma$ induced on $\{z\}\cup W\cup W'$ relative to the partition $\{\{z\},W, W'\}$ is the matrix  
	\[
	M:=\begin{pmatrix}
		0 & |W| &0\\
		1 & |W|-1 & 41\\ 
		0 & |W| & 40
	\end{pmatrix}.
	\]
By Lemma \ref{lem: interlacing} {\rm (ii)} and Lemma \ref{lem: interlacing} {\rm (i)},  the smallest eigenvalue of $M$ is at least $-3$. Thus $\det({M+3I})=-37|W|+258\geq 0$, that is, $|W|\leq 6$. By \eqref{eq size of N(x,z)}, we obtain 
\begin{equation}\label{eq N(x,z) in g} 
	|V_{\mathrm{slim}}(\mathfrak{g})\cap N_{\mathfrak{h}}^{\mathrm{slim}}(x,z)|\geq3.
	\end{equation}

For any $w\in V_{\mathrm{slim}}(\mathfrak{g})$, if $w\in N_{\mathfrak{h}}^{\mathrm{slim}}(x,z)$, then $N_{\mathfrak{h}}^{\mathrm{fat}}(w)\subseteq N_{\mathfrak{h}}^{\mathrm{fat}}(z)$ by Claim $2$, as $w$ and $z$ are adjacent in $\mathfrak{h}$. If $N_{\mathfrak{h}}^{\mathrm{fat}}(w)\neq\{f_2\}$, then $z\in \bigcup\nolimits_{f\in F\cup\{f_1\}}N_{\mathfrak{h}}^{\mathrm{slim}}(f_2,f)$, which contradicts the choice of $z$. So 
$N_{\mathfrak{h}}^{\mathrm{fat}}(w)=\{f_2\}$. This leads to $ V_{\mathrm{slim}}(\mathfrak{g})\cap N_{\mathfrak{h}}^{\mathrm{slim}}(x,z)=V_{\mathrm{slim}}(\mathfrak{g})\cap N_{\mathfrak{g}}^{\mathrm{slim}}(x,f_2)$. By \eqref{eq N(x,z) in g}, $|V_{\mathrm{slim}}(\mathfrak{g})\cap N_{\mathfrak{g}}^{\mathrm{slim}}(x,f_2)|\geq 3$ follows. 
	
%
%
	Our claim follows immediately, if we choose two different vertices distinct from $y$ in the set $V_{\mathrm{slim}}(\mathfrak{g})\cap N_{\mathfrak{g}}^{\mathrm{slim}}(x,f_2)$. This completes the proof of Claim \ref{cla: Dn like}.
\end{proof}

\begin{proof}[\rm\textbf{Proof of Theorem \ref{1 integrability}}]
	By Theorem \ref{c equals k or k-1}, only the case where $c\leq k-2$ should be considered. As $k-c\geq 2$, $n-k-1>2$ follows. Now let $K_3=\max\{K_1(3),K_4,106\}$, where $K_1(3)$ and $K_4$ are such that Theorem \ref{kgyy} and Theorem \ref{Hoffman graph} hold respectively. By Theorem \ref{Hoffman graph integrable}, there exists a Hoffman graph $\mathfrak{h}$ with  $\Gamma$ as its slim graph which has an integral reduced representation of norm $3$. Lemma \ref{rela} tells that $\mathfrak{h}$ also has an integral representation of norm $3$. It follows immediately that $\Gamma$ is $1$-integrable.   
	\end{proof}

\begin{proof}[\rm\textbf{Proof of Theorem \ref{new}}]
Let $K_2=\max\{K_3,11\}$, where $K_3$ is such that Theorem \ref{1 integrability} holds. It follows from Theorem \ref{c equals k or k-1}, Theorem  \ref{c equals 9} and Theorem \ref{1 integrability} immediately.
\end{proof}

\section*{Acknowledgements}

Q. Yang is partially supported by the Fellowship of China Postdoctoral Science Foundation (No. 2020M671855).

B. Gebremichel is supported by the Chinese Scholarship Council at USTC, China.

J.H. Koolen is partially supported by the National Natural Science Foundation of China (No. 12071454), Anhui Initiative in Quantum Information Technologies (No. AHY150000) and the National Key R and D Program of China (No. 2020YFA0713100).

We are also grateful to Mr.~Kiyoto Yoshino for his comments.
\bibliographystyle{plain}
\bibliography{KGRYY}

\end{document}